\pdfoutput=1
\documentclass[reqno, noend, a4paper,oneside,11pt]{amsart}

\usepackage{amsthm,amsmath,amsfonts,mathrsfs,amssymb,todonotes,verbatim}
\usepackage{algorithm}
\usepackage{algorithmicx, algpseudocode}
\usepackage[T1]{fontenc}   
\usepackage{multirow}
\usepackage{url}
\usepackage{multirow}
\usepackage{rotating}
\usepackage{enumerate}
\usepackage{natbib}
\usepackage{epstopdf}
\usepackage{graphicx}
\usepackage{enumitem}
\usepackage{arydshln}
\usepackage[hang]{caption}
\usepackage[hyperfootnotes=false]{hyperref}
\usepackage{setspace}
\usepackage[textheight=640pt]{geometry}
\usepackage{calc}

\algdef{SE}[DOWHILE]{Do}{doWhile}{\algorithmicdo}[1]{\algorithmicwhile\ #1}%

\theoremstyle{plain}
\newtheorem{theorem}{Theorem}[section]
\newtheorem{corollary}[theorem]{Corollary}

\newtheorem{defn}[theorem]{Definition}
\newtheorem{prop}[theorem]{Proposition}
\newtheorem{lemma}[theorem]{Lemma}
\theoremstyle{definition}
\newtheorem{remark}[theorem]{Remark}

\newtheorem{notation}[theorem]{Notation}
\newtheorem{example}[theorem]{Example}

\usepackage{accents}

\DeclareMathOperator{\Mon}{mon}

\DeclareMathOperator{\ord}{ord}
\DeclareMathOperator{\id}{id}

\DeclareMathOperator{\m}{\mathfrak{m}}
\DeclareMathOperator{\jet}{jet}

\DeclareMathOperator{\supp}{supp}
\DeclareMathOperator{\Jac}{Jac}

\DeclareMathOperator{\dash}{-}

\DeclareMathOperator{\N}{\mathbb{N}}

\DeclareMathOperator{\R}{\mathbb{R}}
\DeclareMathOperator{\C}{\mathbb{C}}

\DeclareMathOperator{\coeff}{coeff}

\DeclareMathOperator{\sat}{sat}


\title[Moduli Parameters of Non-Degenerate Newton Boundary Singularities]%
{Moduli Parameters of Complex Singularities with Non-Degenerate Newton Boundary}

\author{Janko B\"ohm}
\address{Janko B\"ohm\\
Department of Mathematics\\
University of Kaiserslautern\\
Erwin-Schr\"o-dinger-Str.\\
67663 Kaiserslautern\\
Germany}
\email{boehm@mathematik.uni-kl.de}

\author{Magdaleen S.\@ Marais}
\address{Magdaleen S.\@ Marais\\
Department of Mathematical Sciences \\
Stellenbosch University\\ and African Institute for Mathematical Sciences\\
P/Bag X1, Matieland 7602, Stellenbosch\\ South Africa\\
}
\email{msmarais@sun.ac.za}

\author{Gerhard Pfister}
\address{Gerhard Pfister\\
Department of Mathematics\\
University of Kaiserslautern\\
Erwin-Schr\"odin\-ger-Str.\\
67663 Kaiserslautern\\
Germany}
\email{pfister@mathematik.uni-kl.de}

\thanks{
This research was supported by 
Project B5 of SFB-TRR 195. Gef\"ordert durch die Deutsche Forschungsgemeinschaft
(DFG) - Projektnummer 286237555 - TRR 195 (Funded by the Deutsche
Forschungsgemeinschaft (DFG, German Research Foundation)~- Project-ID 286237555 - TRR 195).}


\begin{document}

\begin{abstract}

Our recent extension of Arnold's classification includes all singularities of corank $\leq 2$ equivalent to a germ with a non-degenerate Newton boundary, thus broadening the classification's scope significantly by a class which is unbounded with respect to modality and Milnor number. This method is based on proving that all right-equivalence classes within a $\mu$-constant stratum can be represented by a single normal form derived from a regular basis of a suitably selected special fiber. While both Arnold's and our preceding work on normal forms addresses the determination of a normal form family containing the given germ, this paper takes the next natural step: We present an algorithm for computing for a given germ the values of the moduli parameters in its normal form family, that is, a normal form equation in its stable equivalence class. This algorithm will be crucial for understanding the moduli stacks of such singularities. The implementation of this algorithm, along with the foundational classification techniques, is implemented in the library \texttt{arnold.lib} for the computer algebra system \textsc{Singular}.
\end{abstract}

\maketitle

\section{Introduction}\label{Introduction}

Our recent extension \cite{BMP2020} of Arnold's classification of isolated hypersurface singularities (see \cite{A1976, AVG1985}) includes all singularities of corank $\leq 2$ which are equivalent to a germ with non-degenerate Newton boundary in the sense of Kouchnirenko. This broadens the scope of the classification by a class of singularities, which is unbounded both in terms of modality and Milnor number. We establish that there is a single polynomial normal form that contains representatives (at least one, but only finitely many) of all right equivalence classes within a given $\mu$-constant stratum of a given input germ. Based on this result, and algorithmic methods developed in \cite{BMP2016, classify2}, an algorithm for effectively determining a normal form from a regular basis of a suitably chosen special fiber is given.

While both Arnold's and our preceding work on normal forms only addresses the determination of a normal form family containing the given germ, this continuation takes the next natural step:
determining the moduli parameters in the normal form families associated with these input germs. That is, we find for a given input germ an element in the normal form associated to its $\mu$-constant stratum such that this element is right equivalent to the input germ, hence determining exactly its stable and right equivalence class. While one could argue that, with the normal form known, such a germ could be found using an Ansatz for the right equivalence by taking finite determinacy into account, this is practically inefficient and will not yield a result except for trivial input. Taking our clue from some case-by-case studies in \cite{realclassify1, realclassify2, realclassify, BMP2016} we thus develop an iterative method for eliminating the terms of the germ which are not present in the normal form family. This methods has some similarities the algorithm finding the normal form. The main challenge arises here from the fact that singularities with non-degenerate Newton boundary in general do not satisfy Condition A. In an iterative process we thus have to control higher order contributions in the binomial expansion when applying a right equivalence to the germ (since those can be of lower piecewise degree). 
Applications of our result could occur in the context of the study of Baikov polynomials, see \cite{ibp,LP2013}.

Our paper is structured as follows:

In Section \ref{sec1} we give a review of the foundational concepts and preliminary results on singularities and classification.

In Section \ref{sec2}, we recall the main results on the determination of normal forms for singularities of corank $\leq 2$ equivalent to a germ with non-degenerate Newton boundary. We also recall the algorithmic framework determining the normal form.

In Section \ref{secparameters}, we address the determination of the moduli parameters in the normal form corresponding to a given input germ.\smallskip

\emph{Acknowledgements.} Our thanks go to Gert-Martin Greuel, Hans Schönemann for valuable discussions and insights. We also thank Alexander Mathis for work towards an experimental implementation in the computer algebra system \textsc{OSCAR}.

 \section{Definitions and Preliminary Results}\label{sec1}

In this section, fundamental definitions, theorems, and notations relevant to our discussion are presented. We use $\mathbb{C}\{x_1,\ldots,x_n\}$ to denote the ring of convergent power series, that is, power series that converge in open neighborhoods of the point $(0,\ldots,0)$.
We use $\mathfrak{m}$ for the maximal ideal of $\mathbb{C}\{x_1,\ldots,x_n\}$.

\begin{notation}
We denote by $\Mon(x_1,\ldots,x_n)$ the monoid of monomials in $x_1,\ldots,x_n$. For $f \in \mathbb{C}\{x_1,\ldots,x_n\}$ and a monomial $m\in \Mon(x_1,\ldots,x_n)$, the coefficient of $m$ in $f$ is denoted by $\coeff(f,m)$.
\end{notation}

\begin{defn}
If $w=(c_1,\ldots,c_n)\in\N^n$ is a weight for the variables $(x_1,\ldots,x_n)$, the $w$-weighted degree on $\Mon(x_1,\ldots,x_n)$ is defined by the expression \[w\dash\deg\left(\prod_{i=1}^n x_i^{s_i}\right) = \sum_{i=1}^n c_i s_i.\] In the case that the weight of all variables is one, the weighted degree of a monomial $m$ is called its standard degree, denoted by $\deg(m)$. This notation is also used for terms in polynomials.

\end{defn}

\begin{defn}\label{def:piecewiseWeight}
Consider a finite family of weights $w=(w_1,\ldots,w_s)\in(\N^n)^s$ for $(x_1,\ldots,x_n)$. For a term $m\in \C[x_1,\ldots,x_n]$, its \textbf{piecewise weight} with respect to $w$ is defined as
\begin{eqnarray*}
w\dash\deg(m)&:=&\min_{i=1,\ldots,s}w_i\dash\deg(m).\\
\end{eqnarray*}
\end{defn}

\begin{defn}
Fix a (piecewise) weight $w$ on $\Mon(x_1,\ldots,x_n)$.
\begin{enumerate}[leftmargin=10mm]
\item Suppose $$f = \sum_{i = 0}^{\infty} f_{i}$$ is the decomposition of $f \in \C\{x_1,\ldots,x_n\}$ into weighted homogeneous components $f_{i}$ with $w$-degree of $i$. In the case that $f_i=0$ for $i > d$ and $f_d \neq 0$ is the lowest non-zero component, we set $w\dash\deg(f) = d$. The \textbf{(piecewise) weighted $j$-jet} of~$f$, denoted by $w \dash \jet(f, j)$, is given by \[w \dash \jet(f, j) := \sum_{i = 0}^j f_{i} \,.\] The sum of terms of $f$ with the lowest $w$-degree is called the \textbf{principal part} of $f$ with respect to $w$. The \textbf{order} of $f$ with respect to $w$ is defined as the degree of its principal part, and is denoted by $w\dash\operatorname{ord}(f)$.

\item A power series in $\C\{x_1,\ldots,x_n\}$ is said to have \textbf{filtration} $d \in \N$ with respect to $w$ if all its monomials have a $w$-weighted degree $\geq d$. By $E_d^w$ we denote the sub-vector space of $\C\{x_1,\ldots,x_n\}$ of power series of filtration $d$ with respect to $w$. The sub-spaces $E_d^w$, for varying $d \in \N$, form a filtration on $\C\{x_1,\ldots,x_n\}$. 
\end{enumerate}

\end{defn}

\begin{defn}\label{conditionA}
A piecewise homogeneous germ $f_0$ of degree $d$ satisfies Condition A, if for every germ $g$ of filtration $d+\delta>d$ in the ideal spanned by the derivatives of $f_0$, there is a decomposition
\[g=\sum_i\frac{\partial f_0}{\partial x_i}v_i+g',\]
where the vector field $v$ has filtration $\delta$ and $g'$ has filtration bigger than $d+\delta$.
\end{defn}

\begin{defn}
We say that $f\in\m^2\subset \C\{x_1,\ldots,x_n\}$ is \textbf{$k$-determined} if $$ f\sim \jet(f,k)+g\qquad\text{for all } g\in E_{k+1},$$ with respect to right-equivalence. The \textbf{determinacy} of $f$, denoted by $\operatorname{dt}(f)$, is the smallest integer $k$ for which $f$ is $k$-determined.

\end{defn}

\begin{defn}
For $f\in\mathbb C\{x_1,\ldots,x_n\}$, the \textbf{Jacobian ideal} $$\operatorname{Jac}(f)=\left\langle \frac{\partial f}{\partial x_1}, \ldots, \frac{\partial f}{\partial x_n} \right\rangle$$ is the ideal of $\mathbb C\{x_1,\ldots,x_n\}$ generated by the partial derivatives of $f$. The {\bf local algebra} \[Q_f = \mathbb{C}\{x_1, \ldots, x_n\} / \operatorname{Jac}(f)\] of $f$ is the quotient of $\mathbb C\{x_1,\ldots,x_n\}$ by the Jacobian ideal. The \textbf{Milnor number} of $f$ is the dimension of $Q_f$ as a $\mathbb C$-vector space.
\end{defn}

\begin{remark}
If the germ $f$ defines an isolated singularity, then $f$ is $k$-determined if $k\ge \mu(f)+1$, hence $f$ is finitely determined. So an isolated singularity can be represented by a polynomial.
\end{remark}

\begin{defn}
The \textbf{annihilator} of a germ $f$, denoted by $\operatorname{ann}(f)$, is the ideal of all elements of $\mathbb C\{x_1,\ldots,x_n\}$ that yield zero when multiplied with $f$.
\end{defn}

\begin{defn}%\label{phi}
Suppose that  $\phi$ is a $\C$-algebra automorphism of $\C\{x_1,\ldots,x_n\}$, and  $w$ is a single weight on $\Mon(x_1,\ldots,x_n)$.

\begin{enumerate}[leftmargin=10mm]
\item
For any positive integer $j$, the automorphism $w\dash\jet(\phi,j):=\phi_j^w$, is defined by
\[
\phi_j^w(x_i) := w\dash\jet(\phi(x_i),w\dash\deg(x_i)+j) \quad
\text{for }i = 1,\ldots,n.
\]
If $w = (1, \ldots, 1)$, we use the notation $\phi_j$ for $\phi_j^w$.

\item\label{enum:filtration}
We say that $\phi$ has \textbf{filtration} $d$ if
\[
(\phi-\id)E_\lambda^w \subseteq E_{\lambda+d}^w
\]
for all $\lambda \in \N$.
\end{enumerate}

\end{defn}

\begin{remark}
We note that $\phi_0(x_i) = \jet(\phi(x_i), 1)$ for $i = 1, \ldots, n$.
Moreover, note that $\phi_0^w$ has filtration $\le 0$. For $j > 0$, $\phi_j^w$ has filtration $j$ if $\phi_{j-1}^w = \id$.
\end{remark}
\begin{defn}\label{def NB}
 For $f=\sum_{i_1,\ldots,i_n}a_{i_1,\ldots,i_n}x_1^{i_1}\cdots x_n^{i_n}\in\C\{x_1,\ldots,x_n\}$, write
 \begin{eqnarray*}
 \Mon(f)&:=&\{x_1^{i_1}\cdots x_n^{i_n}\ |\ a_{i_1,\ldots,i_n}\neq 0\}\\
 \end{eqnarray*}
 for the set of \textbf{monomials} of $f$, and
 \begin{eqnarray*}
 \sup(f)&:=&\{{i_1}\cdots {i_n}\ |\ a_{i_1,\ldots,i_n}\neq 0\}\\
 \end{eqnarray*}
  for the \textbf{support} of $f$. We set
\begin{eqnarray*}
\Gamma_+(f)&:=&\displaystyle{\bigcup_{x_1^{i_1}\cdots x_n^{i_n}\in\supp(f)}}((i_1,\ldots,i_n)+\R^n_+)\\
\end{eqnarray*}
and define $\Gamma(f)$ as the boundary of the convex hull of $\Gamma_+(f)$ in $\R^n_+$.  The set $\Gamma(f)$ is called the {\bf Newton boundary} of $f$.
Then:
\begin{enumerate}[leftmargin=10mm]
\item The compact segments of $\Gamma(f)$ are referred to as {\bf facets}.\footnote{In convex geometry, codimension $1$ faces of the convex hull $\Gamma_+(f)$ are referred to as facets.} If $\Delta$ is a facet, we write $\supp(f,\Delta)$ for the set of monomials of $f$ with exponent vector on $\Delta$. The sum of the terms lying on $\Delta$ is denoted by $\jet(f,\Delta)$. Moreover, we write $\supp(\Delta)$ for the set of monomials corresponding to the lattice points of $\Delta$. Considering the monomials lying on a union of facets, we use the same terminology for a set of facets.
\item To a facet $\Delta$ we associate a weight $w(\Delta)$ on the monomials in $\Mon(x_1,\ldots,x_n)$ as follows: If $-(w_{x_1},\ldots, w_{x_n})$ is the normal vector of $\Delta$ in lowest terms with integers $w_{x_1},\ldots,w_{x_n}>0$, we define \[w(\Delta)\dash\deg(x_1)= w_{x_1},\ldots, w(\Delta)\dash\deg(x_n)=w_{x_n}.\] 
 \item Now suppose that $w_1,\ldots,w_s$ are the weight vectors of the facets of $\Gamma (f)$ ordered by increasing slope. Then there are uniquely determined minimal integers $\lambda_1,\ldots,\lambda_s\geq 1$ with the property that the piecewise weight with respect to \[w(f):=(\lambda_1 w_1,\ldots,\lambda_s w_s)\] 
 is constant on the Newton boundary $\Gamma(f)$.  We refer to this constant by~$d(f)$.
\item Suppose that  $\Delta_1$ and $\Delta_2$ are adjacent facets with weights $w_1$ and $w_2$, respectively, $w$ is the piecewise weight defined by $w_1$ and $w_2$, and $d$ is the $w$-degree of the monomials on $\Delta_1$ and $\Delta_2$. We then write $\operatorname{span}(\Delta_1,\Delta_2)$ for the Newton polygon associated to the sum of all monomials of $(w_1,w_2)$-degree $d$.
\item If $\Gamma(f)$ has at least one facet, we say that a monomial $m$ is strictly below, on or above $\Gamma(f)$, if the $w(f)$-degree of $m$ is less than, equal to or larger than $d(f)$, respectively. 
\item We write $\jet(f,\Gamma(f))$ for the expansion of $f$ up to $w(f)$-order $d(f)$.
\end{enumerate}
\end{defn}
   
\begin{defn} \label{defn:regularBasis}
Assume that $f$ has finite Milnor number. A basis $\{e_1,\ldots,e_{\mu}\}$ of the local algebra of $f$ consisting out of homogeneous elements is {\bf regular} with respect to the filtration given the piecewise weight $w$, if for each $D\in\mathbb N$, the basis elements of degree $D$ with respect to $w$ are independent modulo the sum vector space $\Jac(f) + E^w_{>D}$ of germs of filtration larger than $D$.
\end{defn}

\begin{remark}
Arnold has proven in \cite{A1974} that for each germ $f\in\mathbb C\{x_,\ldots,x_n\}$ there exists a regular basis consisting %entirely 
out of monomials.
\end{remark}

\begin{defn}\label{def nfequ}
For a union of right-equivalence classes $K\subset \C\{x_1,\ldots, x_n\}$ a {\bf normal form} for $K$ is a smooth map
\[\Phi:\mathcal{B}\longrightarrow \C[x_1,\ldots,x_n]\subset\C\{x_1,\ldots,x_n\}\]
of a finite-dimensional $\C$-linear space $\mathcal{B}$ into the space of polynomials such that:
\begin{itemize}[leftmargin=10mm]
\item[(1)] $\Phi(\mathcal{B})$ intersects all equivalence classes of $K$,
\item[(2)] for each equivalence class the inverse image in $\mathcal{B}$ is finite

\item[(3)] $\Phi^{-1}(\Phi(\mathcal{B})\setminus K)$ is contained in a proper hypersurface in $\mathcal{B}$.
\end{itemize}
We denote the elements of the image of $\Phi$ as {\bf normal form equations}.
\noindent A normal form is called a {\bf polynomial normal form} if the map $\Phi$ is  polynomial.
\end{defn}

\begin{example}
For the germ $f=x^4+y^4$ of Arnold's type $X_9$, the $\mu$-constant stratum of $f$ is covered by the normal form $\Phi: \C\to\C[x,y]$, $\Phi(a)=x^4+ax^2y^2+y^4$.
For instance, the function $g=x^4+\epsilon x^3y+y^4$, with a fixed value of $\epsilon$, lies in the same $\mu$-constant stratum as $f$. Hence, there is a $\C$-algebra isomorphism $\phi_1$ transforming $g$ into $x^4+ax^2y^2+y^4$ with some $a$ in $\C$. As a result, there is also a $\C$-algebra isomorphism $\phi_2$ that maps $g$ to $x^4-ax^2y^2+y^4$.
\end{example}

\begin{defn} (\cite{AVG1985})
Let $f\in\m^2\subset \mathbb C\{x,y\}$ and let $k$ be an upper bound for the determinacy of $f$. The \textbf{modality} of a germ $f\in\m^2\subset \mathbb C\{x,y\}$ is the least number such that a sufficiently small neighborhood of $\jet(f,k)$ in the $k$-jet space can be covered by a finite number of $m$-parameter families of orbits under the right-equivalence action.
\end{defn}

\begin{defn}\label{def:innermodality}\citep{A1974} Let $f\in\m^2\subset\mathbb C\{x,y\}$ be a germ with a non-degenerate Newton boundary. The \textbf{inner modality} is the number of monomials in a regular basis for $Q_f$ lying on or above $\Gamma(f)$.
\end{defn}

\begin{remark}
In the subsequent section, we will recall that the inner modality of a germ $f\in\mathbb C\{x,y\}$ is equal to the number of parameters in the normal form of the germ. Moreover it is shown in \cite{BMP2020}, using results from \cite{G1974}, that the inner modality and modality of a germ with a non-degenerate Newton boundary coincide.
\end{remark}

\section{Classification of Singularities with Non-Degenerate Newton Boundary}\label{sec2}
 In this section we recall the results from \cite{BMP2020}, where it is, in particular, shown that
 \begin{enumerate}
 \item the $\mu$-constant stratum of a germ $f\in\mathbb C\{x,y\}$ with a non-degenerate Newton boundary can be covered, up to right-equivalence, by a single normal form.
 \item there is a normalization condition for the Newton boundary of a germ with a non-degenerate Newton boundary, that ensures the following:  In a $\mu$-constant stratum which contains a germ with non-degenerate Newton boundary, all germs with normalized non-degenerate Newton boundary have the same Newton polygon. Hence, the Newton polygon can be considered as a name of the $\mu$-constant stratum, replacing Arnolds notation of a type. 
 \item there is an effective algorithm to compute the normal form (satisfying the normalization condition) for a given input germ $f\in\mathbb C\{x,y\}$ which is equivalent to germ a non-degenerate Newton boundary.
 \end{enumerate}
 Note that, in this section, we only consider germs in the bivariante convergent power series ring $\mathbb C\{x,y\}$.

The following result from  \cite{GLS2007}, Corollary  2.71, and \cite{BMP2020}, Theorem 3.9, gives a local description of the $\mu$-constant stratum of a germ with a non-degenerate Newton boundary. 

\begin{theorem}\label{corollary:unfolding}
Let $f\in\mathbb C\{x,y\}$ be a germ with a non-degenerate Newton boundary at the origin. A miniversal, equisingular unfolding is given by 
\[F(x,y,{ t})=f+\sum_{i=1}^{m}t_ig_i,\]
where 
$m$ is the modality of $f$, and
 $g_1,\ldots,g_m$ represent a regular basis for $Q_f$ on and above $\Gamma(f)$.
\end{theorem}

A first step to find a normal form for the entire $\mu$-constant stratum of a germ with a non-degenerate Newton boundary is to investigate how a regular basis of the germs in the stratum change, while moving through the stratum.   
\begin{prop}(\cite{BMP2020}, Proposition 3.12)\label{prop:regularBasis}
Let $f_0$ be a 
germ with a non-degenerate Newton boundary 
$\Gamma(f_0)$ and let $f$ be a  
germ with the same Newton polygon as $f_0$ and non-degenerate Newton boundary. Then for $f$ sufficiently close to $f_0$ with respect to the Euclidean distance in the $(\mu+1)$-jet space, the monomials in $\Mon(x,y)$ representing a regular basis for $f_0$ with respect to the  filtration defined by $\Gamma(f_0)=\Gamma(f)$ also represent a regular basis for $f$ with respect to the same filtration.
\end{prop}

Next, it is important to observe that all the germs in the $\mu$-constant stratum of a germ  with a non-degenerate Newton boundary have the same topological type (see \citet{BMP2020}, Remark 3.16). Since germs with a non-degenerate Newton boundary has the same topological type if and only if their characteristic exponents and  intersection numbers coincide (see \citet{BK1986},Theorem 15), and the characteristic exponents and intersection numbers of a germ in $\mathbb C\{x,y\}$ determines the non-degenerate Newton boundaries of a germ that is equivalent to a germ with a non-degenerate Newton boundary (see \citet{BMP2020}, Proposition 4.17 and Corollary 4.18), the next result follows:
\begin{theorem}(\citet{BMP2020}, Theorem 3.18)\label{prop:newtonBoundary}
Suppose $f\in\mathbb C\{x,y\}$ is a convenient germ with non-degenerate Newton boundary~$\Gamma$. Then all the germs in the $\mu$-constant stratum of $f$ are equivalent to a germ with the same Newton polygon $\Gamma$ and  non-degenerate Newton boundary.
\end{theorem}

Let $f$ be a germ with a non-degenerate Newton boundary $\Gamma$. Taking the previous result into account, the next result shows that there exists a set of monomials that is a regular basis for at least one germ, a germ with a non-degenerate Newton boundary and a Newton polygon that coincide with that of $\Gamma$,  in each right-equivalence class of the germs in the $\mu$-constant strantum of $f$.

\begin{lemma}(\citet{BMP2020}, Lemma 3.20)\label{lemma:regularBasis}
Let $f$ be a convenient germ with non-degenerate Newton boundary. Define $f_0$ as the sum of the monomials of $f$ lying on the vertex points of $\Gamma(f)$. Then any regular basis of $f_0$ is also a regular basis for every germ with a non-degenerate Newton boundary in the $\mu$-constant stratum of $f$. 
\end{lemma}

\begin{corollary}(\citet{BMP2020}, Corollary 3.21)\label{corollary:regularBasis}
Suppose $f$ is a convenient germ with non-degenerate Newton boundary. Define $f_0$ as the sum of the monomials of $f$ lying on the vertex points of $\Gamma(f)$, and $f_0'$ as the sum of the terms of $f$ on the vertex points of $\Gamma(f)$. Then any regular basis for $f_0$ is also a regular basis for $f'_0$ and for~$f$.
\end{corollary}

By the next theorem, every germ in the $\mu$-constant with non-degenerate Newton boundary can be written in terms of its Newton boundary and a regular basis of the germ.

\begin{prop}(\cite{BGM2011}, Corollary 4.6)\label{corollary:Markwig}
Let $f\in\mathbb C\{x,y\}$ be a convenient germ with a non-degenerate Newton boundary. Let $f_0$ be the principal part of $f$ and let $\{e_1,\ldots,e_n\}$ be the set of all monomials in a regular basis for $f_0$ lying above $\Gamma(f_0)$. Then there are $\alpha_i$ such that

\[f\sim f_0+\sum_{i=1}^{n} \alpha_ie_i.\]

\end{prop}

Using Theorem \ref{corollary:unfolding}, Proposition \ref{prop:regularBasis}, Theorem \ref{prop:newtonBoundary}, Lemma \ref{lemma:regularBasis} and Proposition \ref{corollary:Markwig} the following theorem can be proved (see \cite{BMP2020}, Theorem 3.22).

\begin{theorem}\label{theorem:normalform}
 Suppose $f$ is a convenient germ with non-degenerate Newton boundary. Define $f_0$ as the sum of the monomials of $f$ lying on the vertex points of  $\Gamma(f)$, and let  $\{e_1,\ldots,e_n\}$ be the set of all monomials in a regular basis for $f_0$ lying on or above $\Gamma(f)$.  Then the family
 \[f_0+\sum_{i=1}^{n} \alpha_ie_i,\]
defines a normal form of the $\mu$-constant stratum containing $f$. 
Restricting the parameters $\alpha_1,\ldots,\alpha_n$ to values such that every germ $f_0+\sum_{i=1}^{n} \alpha_ie_i$ has a  non-degenerate Newton boundary and the same Newton polygon as that of $f$, we obtain all germs in the $\mu$-constant stratum of $f$.  
\end{theorem}

\begin{remark}
%new
Using Theorem \ref{theorem:normalform} a normal form can be constructed for any germ $f\in\mathbb C\{c,y\}$ with a non-degenerate Newton boundary. Note that the Newton polygon of $f$ fixes the Newton polygon of all the germs in the constructed normal form. Since the normal form constructed in Theorem \ref{theorem:normalform} is a normal form for the full $\mu$-constant stratum, it follows that if $f'$ is a germ with a non-degenerate Newton boundary and a  different Newton polygon than $f$, then the normal form constructed for $f'$ describes the same $\mu$-constant stratum as that for $f$. In fact, by Theorem \ref{prop:newtonBoundary}, $f$ is equivalent to a germ with a non-degenerate Newton boundary and same Newton polygon as $f'$. Hence, the normal form for the $\mu$-constant stratum of $f$ as constructed using Theorem \ref{theorem:normalform}, depends on the choice of non-degenerate Newton boundary of germs in the equivalence class of $f$ and the choice of regular basis for the chosen $f_0$. In his lists of normal forms, Arnold associate a \textbf{type} $T$ to each $\mu$-constant stratum. He then fixes a Newton polygon and a choice of moduli monomials (which boils down to the choice of regular basis for $f_0$ in Theorem \ref{theorem:normalform}). For distinguishing between different types, it is sufficient to know the Newton polygon of the normal form.
\end{remark}

To achieve uniqueness of the Newton polygon associated to a fixed type (in order to label types by Newton polygons), a normalization condition on the Newton boundary of a germ with a non-degenerate Newton boundary is needed. Such a condition ensures that the same Newton polygon for any germ in the $\mu$-constant stratum is consistently chosen in order to construct a normal form by using Theorem \ref{theorem:normalform}. 

It is important to distinguish between smooth and non-smooth facets:

\begin{defn}
  A facet of the Newton polygon of a germ is called a \textbf{smooth} if the saturation of its jet is smooth.
  \end{defn}

\begin{defn}
\label{def normal cond}
Suppose $f\in\m^2\subset\mathbb C\{x,y\}$ is a convenient germ with non-degenerate Newton boundary. Let $\Delta$ be a facet of $\Gamma(f)$, and write $w=w(\Delta)$. Then $\jet(f,\Delta)$ factorizes in $\mathbb C[x,y]$ as $$\jet(f,\Delta)=x^a \cdot y^b 
\cdot g_1\cdots g_n
\cdot \widetilde{g},$$ where $a$,$b$ 
are integers, $g_1,\ldots,g_n$ are linear homogeneous polynomials not associated to $x$ or $y$, and $\widetilde{g}$ is a product of non-associated irreducible non-linear homogeneous polynomials. We say that $f$ is \textbf{normalized} with respect to the facet $\Delta$, if 
\[%
\begin{tabular}
[c]{ccccc}%
$w(x)=w(y)$ &  &  & $\Longrightarrow$ & $a,b\neq0$\\
$w(x)>w(y)$ & \text{and} & $a=0$ & $\Longrightarrow$ & $n=0$\\
$w(x)<w(y)$ & \text{and} & $b=0$ & $\Longrightarrow$ & $n=0$%
\end{tabular}
\]
\end{defn}

A germ for which all the facets satisfy the above normalization condition is not necessarily convenient. We can transform such a germ to a convenient germ in the same right-equivalence class by adding the terms $x^{d}$ or $y^{d}$ with $d=\mu(f)+2$, if needed. This will create non-normalized smooth facets that cut the coordinate axes. We  address this in the following definition:

\begin{defn}
% new
A germ $f\in\mathfrak{m}^2\subset\mathbb C\{x,y\}$ satisfies the \textbf{normalization condition} if all of its facets, except smooth facets cutting the coordinate axes, are normalized, and each of its smooth facets that cut a coordinate axis, cut the axis in standard degree $d$, where $d=\mu(f)+1$. 
\end{defn}

It directly follows from Theorem \ref{theorem:normalform} that two germs with the same normalized Newton boundary have the same normal form and hence lie in the same $\mu$-constant stratum. The following theorem states that the normalization condition is reasonable:

 \begin{theorem}In a $\mu$-constant stratum which contains a germ with a non-degenerate Newton boundary,  every right-equivalence class contains a normalized germ, and all germs in the $\mu$-constant stratum  satisfying the normalization condition have the same Newton polygon (up to permutation of the variables).
 \end{theorem} 

In \cite{BMP2020}, algorithms are given to determine a normal form for the $\mu$-constant stratum of a germ which is equivalent to a germ with a non-degenerate Newton boundary. This algorithm also detects if the given input germ is not equivalent to one with a non-degenerate Newton boundary. 

As a first step, this algorithm uses Algorithm 4 of \cite{BMP2020} to transform an input germ $f\in\mathfrak{m}^2\subset\mathbb C\{x,y\}$ to a  germ which is right-equivalent and has a normalized non-degenerate Newton boundary (this step will detect non-degeneracy).

Knowing the normalized non-degenerate Newton boundary, Algorithm 6 of \cite{BMP2020} is used to find a regular basis for the sum of the vertex monomials of the non-degenerate Newton polygon. 

Finally, Algorithm 7 applies Theorem \ref{theorem:normalform} to construct a normal form for the $\mu$-constant stratum of the input germ $f$.

Building on this construction and its implementation in \cite{arnoldlib}, the subsequent section will address finding the moduli parameters corresponding to the given input germ in the normal form.

\section{Determining the Values of the Moduli Parameters in the Normal From of a germ with a Non-Degenerate Newton Boundary}\label{secparameters}

After finding the normal form as discussed in Section~\ref{sec2}, the values of the moduli parameters can, in theory, be computed via an Ansatz for the right equivalence mapping the given germ under consideration to an element of the normal form (making use of finite determinacy). However, this is not practicable except for very small examples.  In this section, we discuss an efficient algorithm for this problem. For brevity of presentation, we introduce the following shorthand notation: if \(\Delta\) is a set of facets of the Newton polygon of \(f\), we write \(f_{\Delta} = \jet(f, \Delta)\) as introduced in Definition~\ref{def NB}.

Now, let $f\in \mathfrak{m}^2$ be a polynomial with a normalized non-degenerate Newton boundary, with $w(f)=(w_1,\ldots,w_n)$ the induced weight on the Newton polygon. Let $f_0$ be the sum of the vertex monomials of $\Gamma(f)$ and write $f'_0=w(f)\dash\jet(f,d(f))$ for the sum of the terms of $f$ on $\Gamma(f)$.

Recall from Lemma \ref{lemma:regularBasis} and Corollary \ref{corollary:regularBasis} that a regular basis $B$ for $f_0$ is also a regular basis for $f'_0$ and for $f$.
Assume that $f$ has a term above $\Gamma(f)$ not in $B$, and let $d'$ be the lowest $w(f)$-degree occurring among these terms. Let $t$ be a term of piecewise degree $d'$ in $f$.
Note that by the properties of a regular basis we can write
\begin{equation}\label{eqp}
t=g\frac{\partial{f}}{\partial{x}}+h\frac{\partial{f}}{\partial{y}}+\text{terms of $w(f)$-degree $d'$ in $B$}+\text{terms of $w(f)$-degree $>d'$,}
\end{equation}
where $g,h\in\mathbb C[x,y]$. 
Define the right equivalence $\phi:\mathbb C[[x,y]]\to\mathbb C[[x,y]]$ by 
\begin{equation}\label{phi}
\phi(x)=x-g,\  \phi(y)=y-h,
\end{equation}
where $g,h$ are as in equation (\ref{eqp}). 
Note that
\begin{eqnarray}\label{eq:erase}
\phi(f)=f-\hspace{-1cm}\underbrace{\left(\frac{\partial f}{\partial x}g+\frac{\partial f}{\partial y}h\right)}_{\text{first order of the binomial expansion of $\phi(f)$}}+\frac{1}{2}\underbrace{\left(\frac{\partial^2 f}{\partial^2 y}h^2+\frac{\partial^2 f}{\partial x\partial y}gh+\frac{\partial^2 f}{\partial^2 x}g^2\right)+\cdots}_{\text{higher order of the binomial expansion of $\phi (f)$}}
\end{eqnarray}

A germ with a non-degenerate Newton boundary does not necessarily satisfy Condition A. Hence we cannot be certain that terms in the higher-order terms in the binomial expansion of $\phi(f)$ are of $w(f)$-degree larger than $d'$. Thus the method introduced in \cite{BMP2020} for finding a normal form equation in general cannot be applied.  Algorithm \ref{alg:normalformeq} provides a method applicable to any germ with a non-degenerate Newton boundary. 

We rely on the following lemma to formulate the algorithm.
Fix a set $\Gamma'$  of connected facets  of the Newton polygon of $f$.

\begin{lemma}\label{zeroSum}
Suppose that \( a \) and \( b \) are the maximal exponents such that the monomial \( x^ay^b \) divides both \( \frac{\partial f_{\Gamma'}}{\partial x} \) and \(  \frac{\partial f_{\Gamma'}}{\partial y} \), and define the exponents $a'$ and $b'$ by $f_{\Gamma'}=x^{a'}y^{b'}\sat(f_{\Gamma'})$.

If $0\neq g,h\in\mathbb C[x,y]$ satisfy 
$$g\cdot\frac{\partial f_{\Gamma'}}{\partial x}+h\cdot\frac{\partial f_{\Gamma'}}{\partial y}=0,$$ 
then 
$$\left(\frac{g}{x^s},\frac{h}{y^t}\right)$$
is a syzygy of 
$$\left(x^i \cdot\sat\left(\frac{\partial f_{\Gamma'}}{\partial x}\right), y^\nu\cdot\sat\left(\frac{\partial f_{\Gamma'}}{\partial y}\right)\right),$$
where 
$s,i,t,\nu$ are given by
\begin{align*}
s &= 
  \begin{cases} 
   \max\{\alpha \mid x^\alpha $ divides $\frac{\partial f_{\Gamma'}}{\partial y}\} & \text{if } a = 0,\\
   1 & \text{if } a \neq 0,
  \end{cases}
&i &= 
  \begin{cases} 
   \max\{\alpha \mid x^\alpha $ divides $\frac{\partial f_{\Gamma'}}{\partial x}\} & \text{if } a = 0,\\
   0 & \text{if } a \neq 0,
  \end{cases}\\
t &= 
  \begin{cases} 
    \max\{\beta \mid y^\beta$ divides $\frac{\partial f_{\Gamma'}}{\partial x}\} & \text{if } b = 0,\\
   1 & \text{if } b \neq 0,
  \end{cases}
&\nu &= 
  \begin{cases} 
   \max\{\beta \mid y^\beta$ divides $\frac{\partial f_{\Gamma'}}{\partial y}\} & \text{if } b = 0,\\
   0 & \text{if } b \neq 0,
  \end{cases}
\end{align*}
Moreover, the vector $\left(\frac{g}{x^s},\frac{h}{y^t}\right)$ is a polynomial multiple of the Koszul syzygy of $$\left(x^i \cdot\sat\left(\frac{\partial f_{\Gamma'}}{\partial x}\right),\, y^\nu\cdot\sat\left(\frac{\partial f_{\Gamma'}}{\partial y}\right)\right).$$
\end{lemma}

We postpone the proof of the lemma to the end of this section. We only remark at the current point that $x^i\sat\left(\frac{\partial f_{\Gamma'}}{\partial x}\right)$, $y^\nu\sat\left(\frac{\partial f_{\Gamma'}}{\partial y}\right)$ is a regular sequence, hence the vector $\left(\frac{g}{x^s},\frac{h}{y^t}\right)$ is a polynomial multiple of the mentioned Koszul syzygy.

\begin{defn}
Let $f\in\mathbb C[x,y]$ and suppose $m=\frac{m_1}{m_2}\in\operatorname{Quot}(\mathbb C[x,y])$ with monomial $m_1, m_2\in\mathbb C[x,y]$, that is, $m$ is a \textbf{Laurent monomial} in $x,y$. Then $\overline{f}^{m}$ is defined as the sum of all terms $t$ of $f$ such that $m\cdot t\in \mathbb C[x,y]$. 
\end{defn}

 Using the next result, we will be able to show in the proof of Algorithm \ref{alg:normalformeq} that higher order terms $t'$ of the binomial expansion of $\phi(f)$, where $\phi$ is defined in (\ref{phi}), can be written as \begin{equation*}
t'=g'\frac{\partial{f}}{\partial{x}}+h'\frac{\partial{f}}{\partial{y}}+\text{terms of $w(f)$-degree $d'$ in $B$}+\text{terms of $w(f)$-degree $>d'$,}
\end{equation*}
where $g',h'\in\mathbb C[x,y]$. Moreover we will show that the transformation $\phi_{\text{new}}:\mathbb C[x,y]\to\mathbb C[x,y]$ defined by $\phi_{\text{new}}(x)=x-g'$, $\phi_{\text{new}}(y)=y-h'$ has a higher filtration than $\phi$.
\begin{theorem}\label{thm:cancellation}
Let $I$ be the ideal generated by all the monomials of $w$-degree $d'$ with $d'\ge d(f)$ fixed. If $(a,b)$ is a syzygy of $(\frac{\partial f}{\partial x},\frac{\partial f}{\partial y})$  over $\mathbb C[x,y]/I$, 
then 
there exist Laurent monomials $z_j$ in $x,y$ such that
$$a-\sum_{j=1}^k z_j\overline{\frac{\partial f}{\partial y}}^{z_j}\in\operatorname{Ann}\left(\frac{\partial f}{\partial x}\right)\text{ \hspace{3mm}and\hspace{3mm} }b+\sum_{j=1}^k z_j\overline{\frac{\partial f}{\partial x}}^{z_j}\in\operatorname{Ann}\left(\frac{\partial f}{\partial y}\right)$$ over $\mathbb C[x,y]/I$. %where the $z_j$ are of the form $\frac{m_1}{m_2}\in\operatorname{Quot}(\mathbb C[x,y])$, $m_1,m_2$ monomials in $\mathbb C[x,y]$. 

Furthermore if no $y^\alpha$, $\alpha>0$, is a monomial of $\overline{\frac{\partial f}{\partial y}}^{z_j}$,  then $x$ divides all terms of $z_j \overline{\frac{\partial f}{\partial y}}^{z_j}$ that are not in $\operatorname{Ann}(\frac{\partial f}{\partial x})$ and do not get cancelled in the sum $\sum_{j=1}^k z_j \overline{\frac{\partial f}{\partial y}}^{z_j}$. Similarly, if no $x^\beta$, $\beta>0$, is a monomial of $\overline{\frac{\partial f}{\partial x}}^{z_j}$, then $y$ divides all terms of $z_j \overline{\frac{\partial f}{\partial x}}^{z_j}$ that are not in $\operatorname{Ann}(\frac{\partial f}{\partial y})$ and do not get cancelled in the sum $\sum_{j=1}^k  z_j \overline{\frac{\partial f}{\partial x}}^{z_j}$.
\end{theorem}

\begin{proof}
Let $(a,b)$ be a syzygy of $\left(\frac{\partial f}{\partial x}, \frac{\partial f}{\partial y}\right)$ in $\mathbb{C}[x,y]/I$, where $I$ is the ideal generated by all the monomials of $w$-degree $d'$. If the syzygy equation $g_1=a\frac{\partial f}{\partial x}+b\frac{\partial f}{\partial y}=0$ has no cancellation below $w$-degree $d'$ then we are finished. Now suppose $g_1$ has cancellation below degree $d'$. Suppose in particular $g_1$ has cancellation below $w_1$-degree $d'$. Let $$l_1 = \min\left(w_1\dash\ord(a\frac{\partial f}{\partial x}),w_1\dash\ord(b\frac{\partial f}{\partial y})\right).$$
For a face $\Delta$ of the Newton polygon, we write$$f_{\Delta,x}=\frac{\partial \jet(f,\Delta)}{\partial x}\text{ \hspace{3mm}and\hspace{3mm} }f_{\Delta,y}=\frac{\partial \jet(f,\Delta)}{\partial y}.$$
Then 
\[w_1\dash\jet(g_1,l_1)=m^{(1)}\left(\underbrace{x^{s_1}y^{\nu_1}\sat\left(f_{\Delta_1,y}\right)}_{w_1\dash\jet(a,l_1-d(f))}f_{\Delta_1,x}\underbrace{-y^{t_1}x^{i_1}\sat\left(f_{\Delta_1,x}\right)}_{w_1\dash\jet(b,l_1-d(f))}f_{\Delta_1,y}\right),\]
where $m^{(1)}$ is a monomial, $\Delta_1$ is the face with the smallest slope of $\Gamma(f)$, and $i_1$, $\nu_1$, $s_1$ and $t_1$ is as in Lemma \ref{zeroSum}.

Now consider the syzygy $\left(\frac{\partial f}{\partial y},-\frac{\partial f}{\partial x}\right)$ of $\left(\frac{\partial f}{\partial x}, \frac{\partial f}{\partial y}\right)$ in $\mathbb{C}[x,y]$, with syzygy equation $g_0=\frac{\partial f}{\partial y}\frac{\partial f}{\partial x}-\frac{\partial f}{\partial x} \frac{\partial f}{\partial y}=0$. Let $l_0=w_1\dash\ord(\frac{\partial f}{\partial y}\frac{\partial f}{\partial x})$, then 
\[w_1\dash\jet(g_0,l_0)=m_{w_1}\left(x^{s_1}y^{\nu_1}\sat\left(f_{\Delta_1,y}\right)f_{\Delta_1,x}-y^{t_1}x^{i_1}\sat\left(f_{\Delta_1,x}\right)f_{\Delta_1,y}\right),\] where $m_{w_1}$ is the product of the maximal power of $x$ and the maximal power of $y$ dividing both $f_{\Delta_1,x}$ and $f_{\Delta_1,y}$. Let $n^{(1)}=\text{lcm}(m_{w_1},m^{(1)})$. Then the lowest nonzero $w_1$-jet of $\frac{n^{(1)}}{m^{(1)}}\left(a,b\right)$ and $\frac{n^{(1)}}{m_{w_1}}\left(\frac{\partial f}{\partial y},\frac{\partial f}{\partial x}\right)$ coincide.

Then
\[(a^{(1)},b^{(1)})=\left(\frac{n^{(1)}}{m^{(1)}}a-\frac{n^{(1)}}{m_{w_1}}\frac{\partial f}{\partial y}, \frac{n^{(1)}}{m^{(1)}}b+\frac{n^{(1)}}{m_{w_1}}\frac{\partial f}{\partial x}\right)\] 
is a syzygy of $\left(\frac{\partial f}{\partial x},\frac{\partial f}{\partial y}\right)$ in $\mathbb C[x,y]/I^{(1)}$, where $I^{(1)}=\frac{n^{(1)}}{m^{(1)}}I$.  

Let $d'_2=d'\left(w_1\dash\deg{\left(\frac{n^{(1)}}{m^{(1)}}\right)}\right)$.
Now, if the equation $g_2=a^{(1)}\frac{\partial f}{\partial x}+b^{(1)}\frac{\partial f}{\partial y}$ has any terms of $w_1$-degree less than $d'_2$, then  the lowest non-zero $w_1$-jet of $(a^{(1)},b^{(1)})$ is \[m^{(2)}\left(x^{s_2}y^{\nu_2}\sat\left(f_{\Delta_1,y}\right), -y^{t_2}x^{i_2}\sat\left(f_{\Delta_1,x}\right)\right).\] Let $n^{(2)}=\text{lcm}\left(m^{(2)}, m_{w_1}\right)$, then the lowest non-zero $w_1$-jet of the syzygies $\frac{n^{(2)}}{m^{(2)}}\left(a^{(1)},b^{(1)}\right)$ and $\frac{n^{(2)}}{m_{w_1}}\left(\frac{\partial f}{\partial y},\frac{\partial f}{\partial x}\right)$ coincide. Similarly, as before, we can create the syzygy
\[\left(a^{(2)},b^{(2)}\right)=\left(\frac{n^{(2)}}{m^{(2)}}a^{(1)}-\frac{n^{(2)}}{m_{w_1}}\frac{\partial f}{\partial y}, \frac{n^{(2)}}{m^{(2)}}b^{(1)}+\frac{n^{(2)}}{m_{w_1}}\frac{\partial f}{\partial x}\right)\] in $\mathbb C[x,y]/I^{(2)}$, where $I^{(2)}=\frac{n^{(2)}}{m^{(2)}}I^{(1)}$. 

We can go on with this process until the syzygy equation $g_k=a^{(k-1)}\frac{\partial f}{\partial x}-b^{(k-1)}\frac{\partial f}{\partial y}$ has no terms of $w_1$-degree less than $w_1$-degree $d'_k=d'_{(k-1)}\left(w_1\dash\deg{\left(\frac{n^{(k-1)}}{m^{(k-1)}}\right)}\right)$. We show now that this will eventually happen. Now, $n^{(j)}=\gcd(c,d,e,h)$ and $m^{(j+1)}$ be the the monomial with the maximal $x$- and $y$-power dividing $c-d$ and $e-h$, where $c=\frac{n^{(j+1)}}{m^{(j+1)}}a^{(j)}$, $d=\frac{n^{(j+1)}}{m_{w_1}}\frac{\partial f}{\partial y}$, $e=\frac{n^{(j+1)}}{m^{(j+1)}}b^{(j)}$ and $h=\frac{n^{(j+1)}}{m_{w_1}}\frac{\partial f}{\partial x}$. Since the lowest non-zero $w_1$-jet of $c$ and $d$, and $e$ and $h$, cancel, respectively, in $c-d$ and $e-h$, it follows that $w_1\dash\deg\left(m^{(j+1)}\right)>w_1\dash\deg\left(n^{(j)}\right)$. 
%It hence follows that $w_1\dash\deg(\frac{m^{(j+1)}m^{(j)}n^{(j)}}{n^{(j+1)}n^{(j)}m^{(0)}})>w_1\dash\deg(\frac{m^{(j)}n^{(j)}}{n^{(j)}m^{(0)}})$. This implies that the lowest jet
Now,
\begin{eqnarray*}
&&w_1\dash\deg\left(\frac{n^{(j+1)}}{m_{w_1}}\right)\cdot w_1\dash\deg\left(\frac{\partial f}{\partial y}\right)-d'_{j+1}\\&=&w_1\dash\deg\left(\frac{n^{(j+1)}}{m_{w_1}}\right)\cdot w_1\dash\deg\left(\frac{\partial f}{\partial y}\right)-w_1\dash\deg\left(\frac{n^{(j+1)}}{m^{(j+1)}}\right)d'_j\\
&=&w_1\dash\deg\left(\frac{m^{(j+1)}}{n^{(j+1)}}\frac{n^{(j+1)}}{m_{w_1}}\right)\cdot w_1\dash\deg\left(\frac{\partial f}{\partial y}\right)-d'_j\\
&=&w_1\dash\deg\left(\frac{m^{(j+1)}}{m_{w_1}}\right)\cdot w_1\dash\deg\left(\frac{\partial f}{\partial y}\right)-d'_j\\
&<&w_1\dash\deg\left(\frac{n^{(j)}}{m_{w_1}}\right)\cdot w_1\dash\deg\left(\frac{\partial f}{\partial y}\right)-d'_j.
\end{eqnarray*}

This implies that the difference in the degree of the lowest non-zero $w_1$-jet of $(a^{(j)},b^{(j)})$ and the degree of the lowest order elements in $d'_k$ becomes smaller. Hence eventually the $w_1$-degree of $(a^{(j)},b^{(j)})$ will be $\ge d'_j$. 

Suppose now that $(a^{(j-1)}, b^{(j-1)})$ is such that $g_j$ has no terms below $w_1$-degree $d'_{j-1}$. We now consider the $w_2$-degree of $(a^{(j-1)}, b^{(j-1)})$. We follow the same strategy. Suppose $g_{j}$ has terms of $w_2$-degree less than $d^{(2)}_{j-1}$, where $$d^{(r)}_{k}=w_r\dash\deg\left(\frac{n^{(1)}n^ {(2)}\cdots n^{(k)}}{m^{(1)}m^ {(2)}\cdots m^{(k)}}\right)d',$$ then
let $$l_j=\min\left(w_2\dash\ord(a^{(j-1)}\frac{\partial f}{\partial x}),w_2\dash\ord(b^{(j-1)}\frac{\partial f}{\partial y})\right).$$ Then 
\[w_2\dash\jet(g_{j},l_j)=m^{(j)}\left(x^{s_j}y^{\nu_j}\sat\left(f_{\Delta_2,y}\right)f_{\Delta_2,x}-y^{t_j}x^{i_j}\sat\left(f_{\Delta_2,x}\right)f_{\Delta_2,y}\right),\]
where $m^{(j)}$ is a monomial, $\Delta_2$ is the face with the second smallest slope of $\Gamma(f)$, and $s_j$, $\nu_j$, $i_j$ and $t_j$ is as in Lemma \ref{zeroSum}. Furthermore 
\[w_2\dash\jet(g_0,l_0)=m_{w_2}\left(x^{s_j}y^{\nu_j}\sat\left(f_{\Delta_2,y}\right)f_{\Delta_2,x}-y^{t_j}x^{i_j}\sat\left(f_{\Delta_2,x}\right)f_{\Delta_2,y}\right),\] where $m_{w_2}=\gcd\left(f_{\Delta_2,x},f_{\Delta_2,y}\right)$. With $$n^{(j)}:=\text{lcm}(m_{w_2},m^{(j)})$$ the lowest nonzero $w_2$-jets of $$\frac{n^{(j)}}{m^{(j)}}\left(a^{(j-1)},b^{(j-1)}\right) \text{ and } \frac{n^{(j)}}{m^{w_2}}\left(\frac{\partial f}{\partial y},\frac{\partial f}{\partial x}\right)$$ coincide. Similarly, as before, we can create the syzygy
\[\left(a^{(j)},b^{(j)}\right)=\left(\frac{n^{(j)}}{m^{(j)}}a^{(j-1)}-\frac{n^{(j)}}{m_{w_j}}\frac{\partial f}{\partial y}, \frac{n^{(j)}}{m^{(j)}}b^{(j-1)}+\frac{n^{(j)}}{m_{w_j}}\frac{\partial f}{\partial x}\right)\] in $\mathbb C[x,y]/I^{(j)}$, where $$I^{(j)}=\frac{n^{(j)}}{m^{(j)}}I^{(j-1)}.$$ 

Note that the syzygy equation $g_{j}$ has no terms below $w_2$-degree $$d'_j:=w_2\dash\deg\left(\frac{n^{(j)}}{m^{(j)}}\right)d^{(2)}_{j-1}$$ and $w_1$-degree $$d_j'':=\left(w_1\dash\deg(\frac{n^{(j)}}{m^{(j)}})\right)\cdot d'_{j-1}.$$ If this would not be the case, this would imply that $$-\frac{n^{(j)}}{m_{w_2}}\frac{\partial f}{\partial y}\frac{\partial f}{\partial x}+\frac{n^{(j)}}{m_{w_2}}\frac{\partial f}{\partial x}\frac{\partial f}{\partial y}$$ has terms below $w_1$-degree $d_j''$. %$$(w_1\dash\deg(\frac{n^{(j)}}{m^{(j)}}))\cdot d'_{j-1}.$$ 

Let $l'_{(x,w_1)}$ and  $l'_{(y,w_1)}$ be the lowest non-zero $w_1$-orders of $\frac{\partial f}{\partial x}$ and $\frac{\partial f}{\partial y}$, respectively. Let $l'_{(x,w_2)}$ and $l'_{(y,w_2)}$ be the lowest $w_2$-orders of $\frac{\partial f}{\partial x}$ and $\frac{\partial f}{\partial y}$, respectively. We observe that $$w_1\dash\jet\left(\frac{\partial f}{\partial x},l'_{(x,w_1)}\right) \text{ and } w_2\dash\jet\left(\frac{\partial f}{\partial x},l'_{(x,w_2)}\right)$$ have coinciding terms at a vertex monomial. The same holds true for   $$w_1\dash\jet\left(\frac{\partial f}{\partial y},l'_{(y,w_1)}\right) \text{ and } w_2\dash\jet\left(\frac{\partial f}{\partial y},l'_{(y,w_2)}\right).$$ All the terms of $$w_1\dash\jet\left(\frac{\partial f}{\partial x},l'_{(x,w_1)}\right)$$ have the same $w_1$-degree. The same holds true for $$w_1\dash\jet\left(\frac{\partial f}{\partial y},l'_{(y,w_1)}\right).$$ Similarly, all terms of $$w_2\dash\jet\left(\frac{\partial f}{\partial x},l'_{(x,w_2)}\right)$$ have the same $w_2$-degree, and the same holds true for the terms of $w_2\dash\jet\left(\frac{\partial f}{\partial y},l'_{(y,w_2)}\right).$ Hence, we conclude that $$w_2\dash\jet\left(-\frac{n^{(j)}}{m_{w_2}}\frac{\partial f}{\partial y}\frac{\partial f}{\partial x}+\frac{n^{(j)}}{m_{w_2}}\frac{\partial f}{\partial x}\frac{\partial f}{\partial y},l_j\right)$$  also has terms below $w_1$-degree $d_j''$,
%$$(w_1\dash\deg(\frac{n^{(j)}}{m^{(j)}}))\cdot d'_{j-1},$$ 
which means that $$w_2\dash\jet\left(\frac{n^{(j)}}{m^{(j)}}a^{(j-1)}\frac{\partial f}{\partial x}+\frac{n^{(j)}}{m^{(j)}}b^{(j-1)}\frac{\partial f}{\partial y},l_j\right)$$ has terms below $w_1$-degree $d_j''$. %$$(w_1\dash\deg(\frac{n^{(j)}}{m^{(j)}}))\cdot d'_{j-1}.$$ 
This again implies that $w_2\dash\jet(a^{(j-1)}\frac{\partial f}{\partial x}-b^{(j-1)}\frac{\partial f}{\partial y},l_{j-1})$ has terms below $w_1$-degree $d'_{j-1}$. We conclude that $a^{(j-1)}\frac{\partial f}{\partial x}-b^{(j-1)}\frac{\partial f}{\partial y}$ has terms below $w_1$-degree $d_{j-1}'$, a contradiction.

Continuing as above, we can construct a syzygy $(a^{(k)},b^{(k)})$ of $(\frac{\partial f}{\partial x},\frac{\partial f}{\partial y})$ in $\mathbb C[x,y]/I^{(k)}$, where the syzygy equation $g_k$ has no terms below $w_i$-degree $d^{(i)}_k$, for $i=1,2,\ldots,n$. Let $r_0=\frac{n^{(1)}\ldots n^{(k)}}{m^{(1)}\cdots m^{(k)}}$ and  $r_j=\frac{n^{(j)}\ldots n^{(k)}}{m^{(j+1)}\cdots m^{(k)}m_{w_j}}$, for $j=\{1,\ldots,k-1\}$, and $r_k=\frac{n^{(k)}}{m_{w_{i_k}}}$. By construction 
\begin{equation}
w_i\dash\jet\left(r_0a, d^{(i)}_k-(w_i\dash\ord(\frac{\partial f}{\partial x}))\right)=w_i\dash\jet\left(\sum_{j=1}^k r_j\frac{\partial f}{\partial y},d^{(i)}_k-(w_i\dash\ord(\frac{\partial f}{\partial x}))\right).
\end{equation}
Since $r_0$ divides all the terms on the right side of the equation, it follows that if $\frac{r_j}{r_0}t\not\in\mathbb C[x,y]$, for some term $t$ of $\frac{\partial f}{\partial y}$, then $t$ either gets cancelled on the righthand side of the equation or $r_jt$ is of $w_i$-degree higher than $d^{(i)}_k-(w_i\dash\ord(\frac{\partial f}{\partial x}))$ for all $i$, that is $r_jt\frac{\partial f}{\partial x}$ is of $w_i$-order higher than $d'$ for all $i$. Hence $r_jt\frac{\partial f}{\partial x}$ is contained in $I^{(k)}$. Therefore
\begin{equation}
w_i\dash\jet\left(r_0a, d^{(i)}_k-(w_i\dash\ord(\frac{\partial f}{\partial x}))\right)=w_i\dash\jet\left(\sum_{j=1}^k r_j\overline{\frac{\partial f}{\partial y}}^{z_j},d^{(i)}_k-(w_i\dash\ord(\frac{\partial f}{\partial x}))\right),
\end{equation}
where $$z_j=\frac{m^{(1)}\cdots m^{(j)}}{n^{(1)}\cdots n^{(j-1)}m_{w_j}}.$$
Therefore $$a-\sum_{j=1}^k z_j \overline{\frac{\partial f}{\partial y}}^{z_j}\in\operatorname{Ann}(\frac{\partial f}{\partial y})\text{ \hspace{3mm}and\hspace{3mm} }b+\sum_{j=1}^k z_j\overline{\frac{\partial f}{\partial x}}^{z_j}\in\operatorname{Ann}(\frac{\partial f}{\partial y})$$ over $\mathbb C[x,y]/I$.

We will now show that if $y^\alpha$ is not a term of $\overline{\frac{\partial f}{\partial y}}^{z_j}$, then $x$ divides all terms of $z_j \overline{\frac{\partial f}{\partial y}}^{z_j}$ that are not in $\operatorname{Ann}(\frac{\partial f}{\partial x})$ and do not get cancelled in the sum $\sum_{j=1}^k z_j \overline{\frac{\partial f}{\partial y}}^{z_j}$. 

So suppose that $\overline{\frac{\partial f}{\partial y}}^{z_j}$, has only mixed terms. Let $cx^\alpha y^\beta$, $c\in\mathbb C$ be such a term of $\overline{\frac{\partial f}{\partial y}}^{z_j}$ such that $z_jcx^\alpha y^\beta\not\in\operatorname{Ann}(\frac{\partial f}{\partial x})$ and $z_jcx^\alpha y^\beta$ is not cancelled in the sum $$\sum_{j=1}^k z_j \overline{\frac{\partial f}{\partial y}}^{z_j}.$$ 
Clearly $x^{\alpha-1} y^{\beta+1}$ is a monomial of $\frac{\partial f}{\partial x}$, and there exists a monomial $lx^\eta y^\gamma$, $l\in\mathbb C$ of $\frac{\partial f}{\partial x}$ such that $$z_jcx^\alpha y^\beta x^\eta y^\gamma\not\in I.$$ But then $lx^{\eta+1} y^{\gamma-1}$ is a monomial of $\frac{\partial f}{\partial y}$. Furthermore, notice that $r_j\frac{c\alpha}{\beta+1}x^{\alpha-1}y^{\beta+1}$ is cancelled in $\sum_{j=1}^k r_j\frac{\partial f}{\partial x}$ if and only if $r_jcx^\alpha y^\beta$ is cancelled in $\sum_{j=1}^kr_j\frac{\partial f}{\partial y}$. Since, in addition, $z_j x^{\alpha-1} y^{\beta+1} x^{\eta+1} y^{\gamma-1}\not\in I$, it follows that $x^{\alpha-1} y^{\beta+1}$ is a monomial of $\overline{\frac{\partial f}{\partial x}}^{z_j}$. On the other hand, since $f$ is convenient, it follows that $y^\alpha$, $\alpha>0$, is a monomial of $\frac{\partial f}{\partial y}$, for some $\alpha>0$. Furthermore, since $f$ has a non-degenerate Newton boundary, it follows that if $y^\beta$ is a monomial of $\frac{\partial f}{\partial y}$, $\beta<\alpha$. Hence, since $y^\alpha$ is not a monomial of $\overline{\frac{\partial f}{\partial x}}^{z_j}$, $z_jy^{\alpha}\not\in\mathbb C[x,y]$. This implies that $z_jy^{\beta}\not\in\mathbb C[x,y]$. Therefore $y^\beta$, for any $\beta>0$ is not a monomial of $\overline{\frac{\partial f}{\partial x}}^{z_j}$.
This again implies that $\alpha>1$. Suppose $z_j = \frac{a_j}{b_j}$, then $b_j\mid x^\alpha y^\beta$ and $b_j\mid x^{\alpha-1} y^\beta$. Hence $x\mid z_jx^\alpha y^\beta$. In a similar way, we can prove that if $x^\beta\not\in\overline{\frac{\partial f}{\partial x}}^{z_j}$, then $y\mid z_j\overline{\frac{\partial f}{\partial x}}^{z_j}$, $\beta>0$.

Analogously we can show that, if $x^\beta$ is not a term of $\overline{\frac{\partial f}{\partial x}}^{z_j}$ for all $j$, then $y$ divides all terms of $ z_j\overline{\frac{\partial f}{\partial x}}^{z_j}$ that are not in $\operatorname{Ann}(\frac{\partial f}{\partial y})$ and do not get cancelled. 
\end{proof}

In line \ref{line:normalize} of Algorithm \ref{alg:normalformeq}, we rely on Lemma \ref{lemma:NB} below. Since its statement is quite obvious, we postpone the proof of the lemma to the end of the chapter.

\begin{lemma}\label{lemma:NB}
\label{lem f0 B}Let $\Gamma$ be the Newton polygon of a germ with normalized,
non-degenerate Newton boundary. Let $f_{0}$ be the sum of the monomials
corresponding to the vertices of $\Gamma$, and let $B$ be a regular basis for
$f_{0}$. Then
\begin{enumerate}
\item any monomial corresponding to a lattice point of $\Gamma$ is a
monomial of $f_{0}$ or an element of $B$, and

\item at most two monomials of $f_{0}$ of degree $\leq\operatorname*{dt}%
(f_{0})$ are not in $B$.
\end{enumerate}
\end{lemma}

The above lemma shows that for a germ with a normalized non-degenerate Newton boundary every monomial corresponding to a lattice point of its Newton polygon is a monomial occurring in any normal form of the germ with the same Newton polygon. Furthermore it shows that only two vertex monomials under the determinacy and on the Newton boundary are not parameter monomials in any such normal form. This implies that the Newton boundary of the given germ can be transformed to the Newton boundary of any of its normal form equations by scaling $x$ and $y$, except for terms above the determinacy.

\vspace{0.5cm}
We now prove correctness and termination of Algorithm \ref{alg:normalformeq}. 
\vspace{0.5cm}

\begin{algorithm}
\caption{Determining the moduli parameters in the normal Form of a germ with a normalized non-degenerate Newton boundary}%
\label{alg:normalformeq}
\begin{spacing}{1.05}
\begin{algorithmic}[1]
\Require{A polynomial germ $f\in\mathbb Q[x,y]$, $f\in\m^3$  of corank $2$ with a normalized non-degenerate Newton boundary;
$f_0$, the sum of the vertex monomials of $\Gamma(f)$; a set of monomials
$B=\{b_1,\ldots,b_m\}$ that is the set of all monomials of $w(f)$-degree $\geq d(f)$ in a regular basis for $f_0$}

\Ensure{A normal form of $f$ and a normal form equation equivalent to $f$ such that the normal form equation is a member of the normal form. }
\State Let $d''$ be a determinacy bound for $f$.
\State Let
 $w:=(w_1,\ldots,w_n):=w(f)$.
\State $F:=f_0+\sum_{i=1}^{m}\alpha_i \cdot b_i$.
\State $S:=\Mon(f-w(f)\dash\jet(f,d(f)))$.
\While{$S\not\subset B$}
\State Let $d'$ be the lowest $w(f)$-degree above $d(f)$ with non-zero terms in $f$.
\State Write the sum $q$ of the terms of $w(f)$-degree  $d'$ as  \[q=g\frac{\partial{f}}{\partial{x}}+h\frac{\partial{f}}{\partial{y}}+\text{terms in $B$ of $w(f)$-degree $d'$}+\text{terms of higher $w(f)$-degree than $d'$}.\]\label{transformation}
\State Define $\phi:\mathbb C[x,y]\to\mathbb C[x,y]$, $\phi(x)= x+g$, $\phi(y)= y+h$.\label{line:lincomb}
\State Let $l$ be the filtration of $\phi$.
\State $q:=w(j)\dash \jet(\phi(f)-(f+g\frac{\partial{f}}{\partial{x}}+h\frac{\partial{f}}{\partial{y}}),d')$.
\State $f:=\jet(\phi(f),d'').$
\label{line:h}
\While{$q\neq 0$}\label{secondorder}
\State Write
\label{line:partials}
\begin{eqnarray*}q
&=&\phi_x \cdot \frac{\partial f}{\partial y}+ \phi_y \cdot \frac{\partial f}{\partial x},  \text{ where }  \phi_x,\phi_y\in\langle x,y\rangle^{l+1}.
\end{eqnarray*}
\State\label{defphixphiy} $\phi_x =w\dash\jet(\phi_x,d'),\quad \phi_y =w\dash\jet(\phi_y,d')$
\State\label{phixphiy} Define $\phi:\mathbb C[x,y]\to\mathbb C[x,y]$ by \[\phi(x):=x-w\dash\jet(\phi_x,d'),\quad \phi(y):=y-w\dash\jet(\phi_y,d').\]
\State $q:=\phi(f)-(f+q)$. \State $f:=\jet(\phi(f),d'')$.
\State $l=\min(\ord(\phi_x),\ord(\phi_y)).$
\EndWhile \label{endsecondorder}
\State $S:=\Mon(f-w(f)\dash\jet(f,d(f)))$.
\EndWhile
\State Normalize two terms of $f$ of $w(f)$-degree $d(f)$, of degree $\le d''$, and not in $B$ to coefficient $1$.\label{line:normalize}
\If{ $\Gamma(f)$ does not intersect the $x$-axis} \label{line:intersect1}
$f:=f+x^a$ with $a=\mu(f)+2$.\label{line:intersect2}
\EndIf
\If{ $\Gamma(f)$ does not intersect the $y$-axis } \label{line:intersect1a}
$f:=f+y^a$ with $a:=\mu(f)+2$.\label{line:intersect2a}
\EndIf

\Return $F$, $f$
\end{algorithmic}
\end{spacing}
\end{algorithm}

\begin{paragraph}{\textit{Proof of Algorithm \ref{alg:normalformeq}}}

In the transformation $\phi$ in line \ref{line:lincomb} we know that terms coming from first partial derivatives with degree $\le d'$, not in $B$, cancel, as described in (\ref{eq:erase}). We now discuss the effect of higher order terms in the binomial expansion after applying $\phi$. Now, it follows from Theorem \ref{thm:cancellation} that $\phi(x)=x+\sum_i z_i\overline{\frac{\partial f}{\partial y}}^{z_i}$ and $\phi(y)=y+\sum_iz_i\overline{\frac{\partial f}{\partial x}}^{z_i}$. Note that, since applying any transformation of filtration $<2$ to $f$ will change its non-degenerate Newton boundary to a degenerate boundary, $\phi$ has filtration $\ge 2$. That is $\ord\left(z_i\overline{\frac{\partial f}{\partial y}}^{z_i}\right)\ge 2$ and $\ord\left(z_i\overline{\frac{\partial f}{\partial x}}^{z_i}\right)\ge 2$. We systematically consider the terms coming from the higher order binomial expansions of $\phi(f)$. 
In higher order binomial expansions the terms are of the following form:
\begin{equation}\label{termexpression}
\frac{\partial^n f}{\partial x^{s}\partial y^{n-s}}\cdot\left(z_{i_1}\overline{\frac{\partial f}{\partial y}}^{z_{i_1}}\right)\cdots\left(z_{i_s}\overline{\frac{\partial f}{\partial y}}^{z_{i_s}}\right)\cdot\left(z_{i_{s+1}}\overline{\frac{\partial f}{\partial x}}^{z_{i_{s+1}}}\right)\cdots\left(z_{i_{n}}\overline{\frac{\partial f}{\partial x}}^{z_{i_{n}}}\right), \end{equation}
where $n>1$, $i_j$'s$\in\mathbb Z$, as well as terms in $J$, where $J$ is the ideal generated by all monomials of degree $d'+1$. Let $z_{i_j}=\frac{a_{i_j}}{b_{i_j}}$. We distinguish between the following types of $b_{i_j}$'s:
\begin{enumerate}
\item[(i)] $b_{i_j}=cx^\alpha y^\beta$, $\alpha,\beta>0$; 
\item[(ii)] $b_{i_j}=cy^\beta$, $\beta>0$; 
\item[(iii)] $b_{i_j}=cx^\alpha$, $\alpha>0$; 
\item[(iv)] $b_{i_j}=1.$
\end{enumerate}

\noindent We consider the following cases:
\begin{enumerate}
\item $b_{i_j}=1$ for some $j\le s$;
\item $b_{i_j}=1$ for some $j> s$;
\item all the $b_{i_j}$'s are of type (i), (ii) and (iii), and the number of $b_{i_j}$'s of type (i) and (iii) is $ <s$.
\item all the $b_{i_j}$'s are of type (i), (ii) and (iii), the number of $b_{i_j}$'s of type (ii) $<n-s$, and the number of $b_{i_j}$'s of type (iii) is $<s$.
\item all the $b_{i_j}$'s are of type (i), (ii) and (iii), and the number of $b_{i_j}$'s of type (i) and (ii) is $<n-s$.
\end{enumerate}

Note that if $z_k=\frac{a_k}{b_k}$ and $b_k=1$, then
\begin{equation}\label{bk=1}\frac{\frac{a_k}{b_k}\left(\overline{x^s\cdot y^{n-(s+1)}\frac{\partial^n f}{\partial x^{s}\partial y^{n-s}}}^{z_k}\right)}{x^s\cdot y^{n-(s+1)}}=\frac{{a_k}\left(x^s\cdot y^{n-(s+1)}\frac{\partial^n f}{\partial x^{s}\partial y^{n-s}}\right)}{x^s\cdot y^{n-(s+1)}}=z_k\left(\overline{\frac{\partial^n f}{\partial x^s\partial y^{n-s}}}^{z_k}\right).\end{equation}
Similarly
\[\frac{\frac{a_k}{b_k}\left(\overline{x^{s-1}\cdot y^{n-s}\frac{\partial^n f}{\partial x^{s}\partial y^{n-s}}}^{z_k}\right)}{x^{s-1}\cdot y^{n-s}}=\frac{{a_k}\left(x^{s-1}\cdot y^{n-s}\frac{\partial^n f}{\partial x^{s}\partial y^{n-s}}\right)}{x^{s-1}\cdot y^{n-s}}=z_k\left(\overline{\frac{\partial^n f}{\partial x^s\partial y^{n-s}}}^{z_k}\right).\]

\noindent Furthermore, note that the monomials of $\left(\overline{x^{s-1}\cdot y^{n-s}\frac{\partial^n f}{\partial x^{s}\partial y^{n-s}}}^{z_k}\right)\in\mathbb C[x,y]$ is a subset of the monomials of 
$\overline{\frac{\partial f}{\partial x}}^{z_k}$. Noticing that for all terms $t$
in $\frac{\partial f}{\partial y}-\overline{\frac{\partial f}{\partial y}^{z_k}}$, 
it follows from the proof of Theorem \ref{thm:cancellation} that $a_k\cdot t\in \text{Ann}_{R/J'}(\frac{\partial f}{\partial x})$, where $J'$ is the ideal generated by all the terms of $w$-degree $(d'+w\dash\deg(b_k)+1)$, and $z_k=\frac{a_k}{b_k }$, it follows that
\begin{equation}\label{zk}\left({x^{s-1}\cdot y^{n-s}\frac{\partial^n f}{\partial x^{s}\partial y^{n-s}}}\right)\cdot z_k\overline{\frac{\partial f}{\partial y}}^{z_k}=z_k\left(\overline{{x^{s-1}\cdot y^{n-s}\frac{\partial^n f}{\partial x^{s}\partial y^{n-s}}}}^{z_k}\right)\cdot\frac{\partial f}{\partial y}+J.\end{equation}
In case (1), it hence follows from (\ref{bk=1}) and (\ref{zk}) that
\begin{eqnarray*}
&&\frac{\partial^n f}{\partial x^{s}\partial y^{n-s}}\cdot\sum_{k=1}^s\left(z_{i_k}\overline{\frac{\partial f}{\partial y}}^{z_{i_k}}\right)\cdot\sum_{k=s+1}^{n}\left(z_{i_{k}}\overline{\frac{\partial f}{\partial x}}^{z_{i_{k}}}\right)\\&=&\frac{1}{x^{s-1}\cdot y^{n-s}}\cdot\left(x^{s-1}\cdot y^{n-s}\frac{\partial^n f}{\partial x^{s}\partial y^{n-s}}\right)\sum_{k=1}^s\left(z_{i_k}\overline{\frac{\partial f}{\partial y}}^{z_{i_k}}\right)\cdot\sum_{k=s+1}^{n}\left(z_{i_{k}}\overline{\frac{\partial f}{\partial x}}^{z_{i_{k}}}\right)\\
&=&\underbrace{\left(z_{i_j}\overline{\frac{\partial^n f}{\partial x^{s}\partial y^{n-s}}}^{z_{i_j}}\right)\cdot\sum_{k=1\ldots s,\ k\neq j}\left(z_{i_k}\overline{\frac{\partial f}{\partial y}}^{z_{i_k}}\right)\cdot\sum_{k=s+1}^{n}\left(z_{i_{k}}\overline{\frac{\partial f}{\partial x}}^{z_{i_{k}}}\right)}_{\in\langle x,y\rangle^{l+1}\subset\mathbb C[x,y]}\cdot\frac{\partial f}{\partial y}.
\end{eqnarray*}

\noindent Note that $\left(z_{i_{k}}\overline{\frac{\partial f}{\partial x}}^{z_{i_{k}}}\right), \left(z_{i_{k}}\overline{\frac{\partial f}{\partial y}}^{z_{i_{k}}}\right)\in\langle x,y\rangle^l\subset\mathbb C[x,y]$, for all $k=s+1,\ldots, n$, where $l\ge 2$.

\noindent In case (2) it similarly follows that
\begin{eqnarray*}
&&\frac{\partial^n f}{\partial x^{s}\partial y^{n-s}}\cdot\sum_{k=1}^s\left(z_{i_k}\overline{\frac{\partial f}{\partial y}}^{z_{i_k}}\right)\cdot\sum_{k=s+1}^{n}\left(z_{i_{k}}\overline{\frac{\partial f}{\partial x}}^{z_{i_{k}}}\right)\\
&=&\underbrace{\left(z_{i_j}\overline{\frac{\partial^n f}{\partial x^{s}\partial y^{n-s}}}^{z_{i_j}}\right)\cdot\sum_{k=1}^s\left(z_{i_k}\overline{\frac{\partial f}{\partial y}}^{z_{i_k}}\right)\cdot\sum_{k=s+1,\ldots n,\ k\neq j}\left(z_{i_{k}}\overline{\frac{\partial f}{\partial x}}^{z_{i_{k}}}\right)}_{\in\langle x,y\rangle^{l+1}\subset\mathbb C[x,y]}\cdot\frac{\partial f}{\partial x}.
\end{eqnarray*}

\noindent Next we consider case (3). Note that
\begin{eqnarray*}\left(x^{s-1} y^{n-s}\frac{\partial^n f}{\partial x^{s}\partial y^{n-s}}\right)\cdot z_k\overline{\frac{\partial f}{\partial y}}^{z_k}\cdot z_l\overline{\frac{\partial f}{\partial x}}^{z_l}&=&z_k\left( \overline{x^{s-1} y^{n-s}\frac{\partial^n f}{\partial x^{s}\partial y^{n-s}}}^{z_k}\right)\cdot \frac{\partial f}{\partial y}\cdot z_l\overline{\frac{\partial f}{\partial x}}^{z_l}\\
&=&z_k\left( \overline{x^{s-1} y^{n-s}\frac{\partial^n f}{\partial x^{s}\partial y^{n-s}}}^{z_k}\right)\cdot z_l\overline{\frac{\partial f}{\partial y}}^{z_l}\cdot \frac{\partial f}{\partial x}\end{eqnarray*}
Using the above method, it follows that 
\begin{eqnarray*}
&&\frac{\partial^n f}{\partial x^{s}\partial y^{n-s}}\cdot\sum_{k=1}^s\left(z_{i_k}\overline{\frac{\partial f}{\partial y}}^{z_{i_k}}\right)\cdot\sum_{k=s+1}^{n}\left(z_{i_{k}}\overline{\frac{\partial f}{\partial x}}^{z_{i_{k}}}\right)\\&=&\frac{z_{i_j}}{x^{s-1}y^{n-s}} \left(\overline{x^{s-1} y^{n-s}\frac{\partial^n f}{\partial x^{s}\partial y^{n-s}}}^{z_{i_j}}\right)\sum_{k=1,\ldots,s,\ k\neq j}\left(z_{i_k}\overline{\frac{\partial f}{\partial y}}^{z_{i_k}}\right)\cdot\hspace{-0.2cm}\sum_{k=s+1}^{n}\left(z_{i_{k}}\overline{\frac{\partial f}{\partial x}}^{z_{i_{k}}}\right)\frac{\partial f}{\partial y},
\end{eqnarray*}
where all the $b_{i_j}$'s are ordered that first, $b_{i_j}$'s of type (i) and (iii) occur and then those of type (ii). In other words the last $n-{s+1}$ $b_{i_j}'$s are of type (ii). This means by Theorem \ref{thm:cancellation} that 
\begin{eqnarray*}
&&\frac{\partial^n f}{\partial x^{s}\partial y^{n-s}}\cdot\sum_{k=1}^s\left(z_{i_k}\overline{\frac{\partial f}{\partial y}}^{z_{i_k}}\right)\cdot\sum_{k=s+1}^{n}\left(z_{i_{k}}\overline{\frac{\partial f}{\partial x}}^{z_{i_{k}}}\right)\\&=&\frac{1}{x^{s-1}}\cdot z_{i_s}\left(\overline{x^{s-1}\cdot y^{n-s}\frac{\partial^n f}{\partial x^{s}\partial y^{n-s}}}^{z_{i_s}}\right)\sum_{k=1,\ldots,s-1}\left(z_{i_k}\overline{\frac{\partial f}{\partial y}}^{z_{i_k}}\right)\cdot\sum_{k=s+1}^{n}\underbrace{\frac{\left(z_{i_{k}}\overline{\frac{\partial f}{\partial x}}^{z_{i_{k}}}\right)}{y}}_{\in\mathbb C[x,y]}\frac{\partial f}{\partial y},
\end{eqnarray*}
Since the number of $b_{i_j}'s$ that are of type (ii) is $>n-s$, $b_{i_s}$ is of type (ii). But then, arguing similar as in (\ref{bk=1}), $x^{s-1}|z_{i_s}\left(\overline{x^{s-1}\cdot y^{n-s}\frac{\partial^n f}{\partial x^{s}\partial y^{n-s}}}^{z_{i_s}}\right)$.
Hence (\ref{termexpression}) can be expressed as
\begin{equation}\label{c3}\underbrace{q'\frac{z_{k}\overline{\frac{\partial f}{\partial x}}^{z_k}}{y}}_{\in\langle x,y\rangle^{l+1}}\frac{\partial f}{\partial y},\quad\text{with }q'\in\langle x,y\rangle^2,\end{equation}
and $b_k$ is of type (ii). In case (4) it follows similarly that (\ref{termexpression}) can be expressed as in (\ref{c3}),
where $b_k$ is of type (i) or (ii).

\noindent In case (5), exchanging $x$ and $y$ in (\ref{c3}), we express (\ref{termexpression}) as
\[\underbrace{q'\frac{z_{k}\overline{\frac{\partial f}{\partial y}}^{z_k}}{x}}_{\in\langle x,y\rangle^{l+1}}\frac{\partial f}{\partial x},\quad\text{with }q'\in\langle x,y\rangle^2,\]
where $b_k$ is of type (iii).

\noindent We conclude that $q$ can be expressed as

\begin{equation}
q=\phi_x\cdot\frac{
\partial f
}{\partial x}+\phi_y\frac{\partial f}{\partial y},
\end{equation}
where $\phi_x,\phi_y\in\langle x,y\rangle^{l+1}$. 
Then the higher orders of the binomial expansion of $\phi(f)$ of degree $\le d'$ can be removed by the first order terms of the transformation
$\phi:\mathbb C[x,y]\to\mathbb C[x,y]$ defined by \[\phi_{\text{new}}(x):=x-\phi_x,\  \phi_{\text{new}}(y):=y-\phi_y.\]
In an analogous way one can show that the sum, $q_{\text{new}}$, of higher order terms of the binomial expansion of degree $\le d'$ of $\phi_{\text{new}}(\phi(f))$ can be written in terms of the same formulas, now in terms of $\phi_{\text{new}}(x)=x-\sum_iz'_i\overline{\frac{\partial f}{\partial y}}^{z'_i}$ and $\phi_{\text{new}}(y)=y-\sum_iz'_i\overline{\frac{\partial f}{\partial x}}^{z'_i}$, as above. Note that the filtration of $\phi_{\text{new}}$ is higher than that of $\phi$.
Hence eventually there will be no terms, not in $B$, of $w$-degree $\le d'$.
$\hfill\square$
\end{paragraph}

For many examples the transformations arising from line \ref{transformation} can be chosen in such a way that the higher orders of the binomial expansion of $\phi(f)$ are of degree larger than $d'$. Choosing $g$ and $h$ in such a way, the while-loop from line \ref{secondorder} to \ref{endsecondorder} is redundant. The next example proves that this is unfortunately not in general the case. 

\begin{example}
Let $f=y^{28}+xy^7+x^2y^3+11x^2y^4+x^{22}$. Then $\frac{\partial f}{\partial x}=2xy^3+22xy^4+y^7+22x^{21}$ and $\frac{\partial f}{\partial y}=3x^2y^2+44x^2y^3+7xy^6+28y^{27}$, and a regular basis for $f$ is $x^{22}y, x^{21}y, x^{20}y, x^{19}y, x^{18}y,$ $ x^{17}y,$ $ x^{16}y,$ $x^2y^3$. Note that $x^2y^4$ is the only monomial above the Newton Boundary, with $w(f)$-degree $1008$. We can express $-11x^2y^4$ as follows in terms of the first partial derivatives.
\[-11x^2y^4=-(7xy+28y^{22})\frac{\partial f}{\partial x}+y^2\frac{\partial f}{\partial y}.\] 
The corresponding transformation $\phi$ is given by
\begin{eqnarray*}
\phi(x)&=&x-(7xy+28y^{22})=x-\sum_iz_i\frac{\partial f}{\partial y}=x-\frac{1}{y^5}\overline{\frac{\partial f}{\partial y}}^{\frac{1}{y^5}}+p_x\\
\phi(y)&=&y+y^2=y-\sum_iz_i\frac{\partial f}{\partial x}=y-(-\frac{1}{y^5}\overline{\frac{\partial f}{\partial x}}^{\frac{1}{y^5}})+p_y,
\end{eqnarray*}
where $p_x\frac{\partial f}{\partial x}$ and $p_y\frac{\partial f}{\partial y}$ are of $w(f)$-degree $1008$ or higher. 
Note that the filtration of $\phi$ is $l=2$. Now let
\[q = w\dash\jet(\phi(f)-f-\phi(x)\frac{\partial f}{\partial x}-\phi(y)\frac{\partial f}{\partial y}, 1008).\]
Then $q$ is the contribution of the orders $>1$ in the binomial expansion of $\phi$ in the $1008$-jet of $\phi(f)$. 
\begin{equation*}
q=-28xy^9+182y^{30}
\end{equation*}
To write $q$ as in line \ref{line:partials} in Algorithm \ref{alg:normalformeq},  we are computing $\phi_x$ and $\phi_y$. We do this by forming an ideal $I$ generated by the set 
\[\left\{\left.x^iy^j\frac{\partial f}{\partial x}, x^iy^j\frac{\partial f}{\partial y} \quad\right|\quad i+j=l+1=3\quad\right\},\]
and all monomials of higher piecewise degree than $1008$. We then write $q$ in terms of $\frac{\partial f}{\partial x}$, $\frac{\partial f}{\partial y}$, where the coefficients are of order $3$ or higher,  using a Gr\"obner basis of $I$ with a global ordering, and then lifting the result back to the original generators using the command \texttt{liftstd()} in \textsc{Singular}. It turns out that 
\[q=\left((-28-364y^{17}-4004y^{18})xy^2+182y^{23}\right)\cdot\frac{\partial f}{\partial x}\]
By abuse of notation, let $f=\phi(f)$. We now form a new $\phi$ as described in line \ref{phixphiy} of Algorithm \ref{alg:normalformeq}. We set \[\phi_x=28xy^2-182y^{23}\text{ and }\phi_y=0.\] We define $\phi:\mathbb C[x,y]\to\mathbb C[x,y]$ by $\phi(x)=x-\phi_x$ and $\phi(y)=y-\phi_y=y.$
Hence the first order terms of the binomial expansion of $\phi(f)$ up to degree $1008$ cancel $q$. In this case higher order terms of the binomial expansion of $\phi(f)$ does not have terms of $w(f)$-degree 1008 or lower.
In fact $w\dash\jet(\phi(f),1008)=y^{28}+xy^7+x^2y^3+x^{22}$.\medskip
\end{example}

In the next example more than one iteration of the while-loop in Algorithm \ref{alg:normalformeq} is needed to transform the germ $f$ to its normalform equation up to $w(f)$-degree $700$.
\begin{example}
For $f=x^2y^4+x^4y^2+x^{20}+y^{40}+60x^{21}y^{14}$, we have $\frac{\partial f}{\partial x}=4x^3y^2+2xy^4+20x^{19}+1260x^{20}y^{14}
$ and $\frac{\partial f}{\partial y}=2x^4y+4x^2y^3+840x^{21}y^{13}+40y^{39}$, and a regular basis is $x^4y^4$, $x^4y^3$, $x^3y^4$, $x^4y^2$, $x^3y^3$, $x^2y^4$. Note that $x^{21}y^{14}$ is the only monomial above the Newton boundary with $w(f)$-degree $700$. Now $60x^{21}y^{14}$ can be written as
\begin{eqnarray*}
60x^{21}y^{14}&=&(2x^4y^{12}+4x^2y^{14}+40y^{50}-10x^{20}y^{10})\frac{\partial f}{\partial x}-(4x^3y^{13}+2xy^{15})\frac{\partial f}{\partial y}.
\end{eqnarray*}
The corresponding transformation $\phi$ is given by 
\begin{eqnarray*}
\phi(x)&=&x-(2x^4y^{12}+4x^2y^{14}+40y^{50}-10x^{20}y^{10})=x-y^{11}\overline{\frac{\partial f}{\partial y}}^{y^{11}}\hspace{-0.5cm}-5x^{16}y^9\overline{\frac{\partial f}{\partial y}}^{x^{16}y^9}\hspace{-0.5cm}+p_x\\
\phi(y)&=&y+(4x^3y^{13}+2xy^{15}) = y-(-y^{11}\overline{\frac{\partial f}{\partial x}}^{y^{11}}-5x^{16}y^9\overline{\frac{\partial f}{\partial x}}^{x^{16}y^9})+p_y,
\end{eqnarray*}
where $p_x\frac{\partial f}{\partial x}$ and $p_y\frac{\partial f}{\partial y}$ is of $w$-degree $700$ or higher. Note that the filtration of $\phi$ is $l=16$. Now let
\[q = w\dash\jet(\phi(f)-f-\phi(x)\frac{\partial f}{\partial x}-\phi(y)\frac{\partial f}{\partial y}, 700).\]
Then $q$ is the contribution of the orders $>1$ in the binomial expansion of $\phi$ in the $700$-jet of $\phi(f)$.
\begin{eqnarray*}
q&=&-24x^{10}y^{26}-12x^8y^{28}-12x^6y^{30}-24x^4y^{32}-224x^7y^{44}-32x^5y^{46}+144x^6y^{60}\\&&+12160x^6y^{64}+12640x^4y^{66}+2800x^2y^{68}+384x^7y^{74}+952320x^7y^{78}\\&&+472320x^5y^{80}+79680x^3y^{82}+256x^8y^{88}+11704320x^6y^{94}+1467360x^4y^{96}\\&&+9600x^2y^{102}+1600y^{104}+21065216x^5y^{110}+38400x^5y^{114}-12800x^3y^{116}\\&&+12800xy^{118}+245661440x^6y^{124}+38400x^4y^{130}+38400x^2y^{132}+51200x^3y^{146}\\&&-256000xy^{152}+25600x^4y^{160}+2560000y^{202}.
\end{eqnarray*}

Note that the order of $q$ is $36$.

To write $q$ as in line (\ref{line:partials}) in Algorithm \ref{alg:normalformeq},  we are computing $\phi_x$ and $\phi_y$. We do this by forming an ideal $I$ generated by \[\left\{\left.x^iy^j\frac{\partial f}{\partial x}, x^iy^j\frac{\partial f}{\partial y} \quad\right|\quad i+j=l+1=17\quad\right\},\] and all monomials of higher piecewise degree than $700$. We write $q$ in terms of $\frac{\partial f}{\partial x}$ and $\frac{\partial f}{\partial y}$, by using a Gr\"{o}bner basis of $I$ with a global ordering, and then lifting the result back to the original generators using the command \texttt{liftstd()}. We now form $\phi$ as described in line \ref{phixphiy} of Algorithm \ref{alg:normalformeq}. Let $\phi_x$ and $\phi_y$ be as defined in line(\ref{defphixphiy}) in Algorithm \ref{alg:normalformeq}. It turns out that \begin{eqnarray*}\phi_x &=& 256x^{21}y^{10}+512x^{19}y^{12}+128x^{17}y^{14}+448x^{15}y^{16}+288x^{13}y^{18}-144x^{11}y^{20}\\&&+72x^{9}y^{22}-42x^7y^{24}+18x^5y^{26}-12x^3y^{28}+\frac{12650}{567}x^{34}y^{10}-320x^{10}y^{36}\\&&+160x^8y^{38}-80x^6y^{40}-16x^4y^{42}+288x^9y^{52}-144x^7y^{54}+72x^5y^{56}-4800x^9y^{56}\\&&+2400x^7y^{58}-1200x^5y^{60}+3640x^3y^{62}+1320xy^{64}-384x^8y^{68}+192x^6y^{70}\\&&-326400x^8y^{72}+163200x^6y^{74}+156480x^4y^{76}+39840x^2y^{78}+128x^7y^{84}\\&&-8769600x^7y^{88}+4384800x^5y^{90}+733680x^3y^{92}-38400x^7y^{92}+19200x^5y^{94}\\&&-9600x^3y^{96}+4800xy^98-21065216x^6y^{104}+10532608x^4y^{106}-115200x^6y^{108}\\&&+57600x^4y^{110}-19200x^2y^{112}+6400y^{114}+122830720x^5y^{120}+38400x^5y^{124}\\&&-19200x^3y^{126}+19200xy^{128}-51200x^4y^{140}+25600x^2y^{142}-512000x^4y^{144}\\&&+256000x^2y^{146}-128000y^{148}+12800x^3y^{156}+192000x^3y^{160}-128000xy^{162}\end{eqnarray*}
and
\begin{eqnarray*}
 \phi_y&\hspace{-0.2cm}=&\hspace{-0.2cm}-512x^{20}y^{11}\hspace{-0.1cm}-256x^{18}y^{13}\hspace{-0.1cm}-256x^{16}y^{15}\hspace{-0.1cm}-512x^{14}y^{17}\hspace{-0.1cm}-2560x^{36}y^9+40y^{65}+64000y^{163}.   
\end{eqnarray*}
Note that order of $\phi_x$ and $\phi_y$, and hence the filtration of $\phi$, is $31$, respectively.
 Unfortunately the higher order binomial expansions of $\phi(f)$ again contributes terms of $w$-degree $\le 700$. Note that this time the order of these terms is $66>36$. Now let
\[q = w\dash\jet(\phi(f)-f-\phi(x)\frac{\partial f}{\partial x}-\phi(y)\frac{\partial f}{\partial y}, 700).\]
Then
\begin{eqnarray*}
q&=&1296x^8y^{58}-1008x^6y^{60}-3072x^7y^{74}-7488x^8y^{88}+736800x^6y^{94}+189120x^4y^{96}\\&&+6049280x^5y^{110}+136922880x^6y^{124}-98264000x^4y^{130}-12833600x^2y^{132}\\&&-681446400x^3y^{146}-23950380800x^4y^{160}-77184000x^2y^{166}-5408000y^{168}\\&&-94208000xy^{182}+33001344000x^2y^{196}+13038080000y^{232}.
\end{eqnarray*}
Let $f=\phi(f)$. Repeating the process, we again write $q$ in terms of \[\left\{\left.x^iy^j\frac{\partial f}{\partial x}, x^iy^j\frac{\partial f}{\partial y} \quad\right|\quad i+j=l+1=32\quad\right\},\] and all terms of piecewise degree higher than $700$ to compute $\phi_x$ and $\phi_y$. It turns out that
\begin{eqnarray*}
\phi_x&=&-3312x^9y^{52}+1656x^7y^{54}-504x^5y^{56}+3072x^8y^{68}-1536x^6y^{70}-3744x^7y^{84}\\&&-358560x^7y^{88}+179280x^5y^{90}+94560x^3y^{92}-6049280x^6y^{104}+3024640x^4y^{106}\\&&+68461440x^5y^{120}+73408000x^5y^{124}-36704000x^3y^{126}-6146400xy^{128}\\&&+681446400x^4y^{140}-340723200x^2y^{142}-11975190400x^3y^{156}+77184000x^3y^{160}\\&&-38592000xy^{162}+94208000x^2y^{176}-47104000y^{178}+15848768000xy^{192}
\end{eqnarray*}
and
\begin{eqnarray*}
\phi_y&=&-135200y^{129}+325952000y^{193}. 
\end{eqnarray*}
Applying $\phi$, formed as in line (\ref{phixphiy}), 
\[w\dash\jet(\phi(f),700)=x^4y^2+x^2y^4+x^{20}+y^{40}.\]
\medskip
\end{example}

We now turn to the proof of our two main lemmata, and begin with Lemma \ref{zeroSum}, 

\begin{proof}

We say that terms of \( f_x := \frac{\partial f_{\Gamma'}}{\partial x} \) and \( f_y := \frac{\partial f_{\Gamma'}}{\partial y} \) are in correspondence if they originate from the same term of \( f_{\Gamma'} \).

\begin{itemize}[leftmargin=5mm]
\item If $a\neq 0$, then $x^a$ is the largest power of $x$ dividing $f_x$, and $x^{a+1}$ is the largest power of $x$ dividing $f_y$. To see this, we first note that, excluding a pure $x$-power term in $f_x$, all terms of $f_x$ and $f_y$ are in one-to one correspondence. If $f_x$ would be divisible by $x^{a+1}$ then, since all terms of $f_y$ are in correspondence to a term in $f_x$, all terms of $f_y$ would be divisible by $x^{a+2}$, contradicting the minimal choice of $a$. Thus $x^a$ is the largest power of $x$ dividing $f_x$, and $f_y$ is divisible by $x^{a+1}$. If $x^{a+2}$ would divide $f_y$, then $x^{a+1}$ would also divide every corresponding term of $f_x$. Since also a possible $x^{\alpha}$-term of $f_x$ is then divisible by $x^{a+2}$ (since it is the term with the largest $x$-exponent), this again contradicts the minimal choice of $x$. Moreover the argument above implies that $x^{a+1}$ is the highest power of $x$ dividing $f$. 
\item If $a=0$, then $x\nmid f_y$ or $x\nmid f_x$. 
\begin{itemize}
\item If $x\nmid f_y$ then $f_{\Gamma'}$ has a term that is a pure $y$-power. Hence $x\nmid f$. 
\item If $x|f_y$, then $x\nmid f_x$, which implies that $f_{\Gamma'}$ has a term of the form $xy^l$ with $l\geq 1$. Hence, in this case, if $x^s$ divides $f_y$ (and hence $f$), $s=1$.
\end{itemize}
\end{itemize}

Now, let $x^{a'}y^{b'}$ be the product of the highest power of $x$ and the highest power of $y$ dividing $\jet(f,\Gamma')$. 
\begin{itemize}
    \item Then $a'=a+1$, if $a\neq 0$.
    \item If $a=0$, we have two cases:
    \begin{itemize}
        \item If $x\nmid f_y$, then $a'=0$.
        \item If $x|f_y$, then $1$ is the highest power of $x$ dividing $f_{\Gamma'}$. Hence $a'=a+1$.
    \end{itemize}
\end{itemize}
Similarly $b'=b+1$, if $b\neq 0$. If $b=0$ and  $y\nmid f_x$, then $b'=0$. If $b=0$, $y|f_x$ (implying that $y\nmid f_y$, then $1$ is the highest power of $y$ dividing $f_{\Gamma'}$. Hence $b'=b+1$.

We assume that
\[g\cdot f_x+h\cdot f_y=0.\]
Then 
\[x^ay^b(g\cdot y^t\cdot x^i\cdot \sat(f_x)+h\cdot x^s\cdot y^{\nu}\cdot\sat(f_y))=0,\]
where $s$, $t$, $i$ and $\nu$ are defined as above. 
To see this, consider first the case $a\neq 0$, $b\neq 0$, then
\[x^ay^b(g\cdot y\cdot \sat(f_x)+h\cdot x\cdot\sat(f_y))=x^{a'}y^{b'}\left(\frac{g}{x}\cdot \sat(f_x)+\frac{h}{y}\cdot\sat(f_y)\right)=0,\]
which can be verified by the arguments above. The other cases can be similarly verified.
If $a\neq 0$, or if $a=0$, and $x|f_y$, then $s=1$ and $a'=a+1$. In both these cases $i=0$. If $a=0$ and $x\nmid f_y$, then $a'=a=0$ and $s=0$. Therfore we have in general.
\[x^{a'}y^{b}(g\cdot y^t\cdot x^i\cdot\sat(f_x)+h\cdot x^s\cdot y^\nu\cdot\sat(f_y))=0.\]
Similarly, if $b\neq 0$, or if $b=0$, and $y|f_x$, then $\nu=1$ and $b'=b+1$. In both these cases $t=0$. If $b=0$ and $y\nmid f_x$, then $b'=b=0$ and $t=0$. Therfore we have in general.
\[x^{a'}y^{b'}(\frac{g}{x^s}\cdot x^i\cdot\sat(f_x)+\frac{h}{y^t}\cdot y^\nu\cdot\sat(f_y))=0.\]
Therefore
\[\frac{g}{x^s}\cdot x^i\cdot\sat(f_x)+\frac{h}{y^t}\cdot y^\nu\cdot\sat(f_y)=0.\]
Hence $x|g$, if $a\neq 0$, or if $a=0$ and $x|f_y$, and $y|h$, if $b\neq 0$, or if $b=0$ and $y|f_x$.

We now show that $x^i\cdot\sat(f_x)$ and $y^\nu\cdot\sat(f_y)$ have no common monomial factor.
To do that, first note that $\sat (f_{\Gamma'})$ is nondegenerate by definition. This implies that $\sat (f_{\Gamma'})$ does not have a multiple monomial factor, hence $\sat (f_{\Gamma'})$ and $\frac{\partial \sat f_{\Gamma'}}{\partial x}$ have no common monomial factor, and $\sat (f_{\Gamma'})$ and $\frac{\partial \sat f_{\Gamma'}}{\partial y}$ have no common monomial factor. Thus, $x^{a'}y^{b'}$ accounts for the only multiple factors in $f_{\Gamma'}$. 

Consider the following equations for the partial derivatives of $f_{\Gamma'}$:

\[f_x=a'x^{a'-1}y^{b'}(\sat (f_{\Gamma'}))+x^{a'}y^{b'}(\frac{\partial \sat f_{\Gamma'}}{\partial x})\]

\[f_y=b'x^{a'}y^{b'-1}(\sat (f_{\Gamma'}))+x^{a'}y^{b'}(\frac{\partial \sat f_{\Gamma'}}{\partial y}).\]

If the saturations of the partial derivatives of $f_{\Gamma'}$ with respect to $x$ and $y$ share a common factor, then this factor would also be a factor of $\sat (f_{\Gamma'})$. However, this contradicts the previous equations. Furthermore $y\nmid \sat(f_x)$ if $\nu>0$ and $x\nmid\sat(f_y)$ if $i>0$.

Since $x^i\sat\left(\frac{\partial f_{\Gamma'}}{\partial x}\right)$, $y^\nu\sat\left(\frac{\partial f_{\Gamma'}}{\partial y}\right)$ is a regular sequence, the vector $\left(\frac{g}{x^s},\frac{h}{y^t}\right)$ is a polynomial multiple of the Koszul syzygy of  $(x^i\cdot\sat(\frac{\partial f_{\Gamma'}}{\partial x}),y^{\nu}\cdot\sat(\frac{\partial f_{\Gamma'}}{\partial y}))$. This completes the proof.

\end{proof}

We turn now to the proof of Lemma \ref{lemma:NB}, where we use the following observation:

\begin{remark}
\label{rmk generic}With the notation of Lemma \ref{lem f0 B}, let $m$ be a
monomial corresponding to a lattice point of $\Gamma$. Suppose $g$ is a germ with non-degenerate Newton boundary $\Gamma(g)=\Gamma$ and
with monomial $m$, such that the corresponding term cannot be removed from $g$
by a right equivalence. Then Theorem 3.20 implies that $m$ is a monomial of
$f_{0}$ or $m\in B$.
\end{remark}

We now prove Lemma \ref{lem f0 B}:

\begin{proof}
We prove the first part of the lemma: It is sufficient to show that for any
monomial $m$ in the relative interior of a face $\Delta$ of $\Gamma$, there
exists a germ $g$ as in Remark \ref{rmk generic}. For any germ $g$ with
$\Gamma(g)=\Gamma$, we can write
\[
\jet(g,\Delta)=x^{a}\cdot y^{b}\cdot g_{1}\cdots g_{n}\cdot\widetilde{g}%
\]
where $a$,$b$ are integers, $g_{1},\ldots,g_{n}$ are linear homogeneous
polynomials not associated to $x$ or $y$, and $\widetilde{g}$ is a product of
non-associated irreducible non-linear homogeneous polynomials. Note that $a$
and $b$ are the distances of $\Delta$ from the $y$- and $x$-coordinate axes.
Note also that smooth faces meeting the coordinate axes do not contain
interior lattice points. After possibly exchanging $x$ and $y$, we may assume
that for the weight $w$ of $\Delta$ we have $w(x)\geq w(y)$. First consider
the case that $w(x)>w(y)$:

\begin{itemize}[leftmargin=9mm]
\item[(1.a)] Suppose that $a\geq2$. Let $g$ be a germ non-degenerate Newton
boundary $\Gamma$. Then any right-equivalence which does not act only as a
rescaling of variables on $\jet(g,\Delta)$ generates terms on the coordinate
axes below $\Gamma$, hence, changes the Newton polygon. The claim follows
directly (choosing a non-degenerate $g=f_{0}+c\cdot m$ with $c\in
\mathbb{C}^{\ast}$).

\item[(1.b)] Consider now the case $a=0$. Since $\Gamma$ corresponds to a
normalized germ with respect to $\Delta$, we have $n=0$ and $\widetilde{g}%
\neq1$. This implies that $w(y)\nmid w(x)$, hence there does not exist a
$w$-weighted homogeneous right-equivalence except rescaling of variables.

\item[(1.c)] Finally, suppose $a=1$. If $w(y)\nmid w(x)$, then, as above,
$n=0$ and $\tilde{g}\neq0$ which implies that rescalings are the only
right-equivalences on the face. If $w(y)\mid w(x)$, then, writing
$\tau=w(x)/w(y)$, any right-equivalence which does not create any terms of
lower $w$-weight than that of $\Delta$ is of the form
\[
x\mapsto c_{1}x+c_{2}y^{\tau},\text{ }y\mapsto c_{3}y,
\]
where $c_{1},c_{3}\in\mathbb{C}^{\ast}$ and $c_{2}\in\mathbb{C}$, hence acts
on $y$ as a rescaling. We may therefore assume that $b=0$. The vertices of
$\Delta$ correspond to monomials of the form $x^{p}$ and $xy^{(p-1)\cdot\tau}$
with $p\geq2$. Any monomial in the interior of $\Delta$ is of the form
$m=x^{s}y^{\tau\cdot(p-s)}$ with $0<s<p$. For $g=f_{0}+c\cdot m$ with
$c\in\mathbb{C}^{\ast}$, the jet with regard to $\Delta$ is then
$\jet(g,\Delta)=x^{p}+c\cdot m+xy^{(p-1)\cdot\tau}$. We now show that there is
no right-equivalence which keeps $\Gamma(g)$ and, hence, is of the above form,
that removes $m$. Keeping the face $\Delta$ and removing $m$ amounts to the
conditions%
\begin{align*}
\binom{p}{p-s}c_{1}^{s}c_{2}^{p-s}+c\cdot c_{1}^{s}c_{3}^{\tau\cdot(p-s)}  &
=0\\
c_{2}^{p}+c\cdot c_{2}^{s}c_{3}^{\tau\cdot(p-s)}+c_{2}c_{3}^{(p-1)\cdot\tau}
&  =0
\end{align*}
which correspond to the vanishing of the coefficients of $m$ and
$y^{p\cdot\tau}$. Using $c_{1}\neq0$ a solution of the first equation for
$c_{2}$ is of the form%
\[
c_{2}=\tilde{c}\cdot c_{3}^{\tau}%
\]
where $\tilde{c}^{p-s}=-c/\binom{p}{s}$. Inserting this into the second
equation, leads to the equation%
\begin{align*}
0  &  =\tilde{c}^{p}\cdot c_{3}^{\tau\cdot p}+c\cdot\tilde{c}^{s}c_{3}%
^{\tau\cdot p}+\tilde{c}\cdot c_{3}^{\tau\cdot p}\\
&  =c_{3}^{\tau\cdot p}\cdot(\tilde{c}^{p}+c\cdot\tilde{c}^{s}+\tilde{c})\\
&  =c_{3}^{\tau\cdot p}\cdot\left(  \tilde{c}^{p}-\binom{p}{s}\cdot\tilde
{c}^{p}+\tilde{c}\right) \\
&  =c_{3}^{\tau\cdot p}\cdot\tilde{c}\cdot\left(  \left(  1-\binom{p}%
{s}\right)  \cdot\tilde{c}^{p-1}+1\right)
\end{align*}
Since $\binom{p}{s}\neq1$, there is a Zariski open set of values of $\tilde
{c}$, equivalently of $c$, such that the expression in the bracket does not
vanish and $g$ is non-degenerate. For such a choice of $c$ and thus of $g$, it
follows that $c_{3}=0$, a contradiction.
\end{itemize}\medskip

Now consider the case that $w(x)=w(y)$. Since $\jet(g,\Delta)$ is homogeneous,
hence factorizes into linears, in this case we have $a\geq1$, $b\geq1$ and
$\tilde{g}=0$.

\begin{itemize}[leftmargin=9mm]
\item[(1.d)] If $a,b\geq2$, and $g$ is a germ with non-degenerate Newton
boundary and Newton polygon $\Gamma$, then any right-equivalence which does
not only act as a rescaling of variables on $\jet(g,\Delta)$ changes the
Newton polygon (with the same argument as in the case $w(x)>w(y)$, $a\geq2$).

\item[(1.e)] The case $a\geq2$ and $b=1$, can be handled in the same way as
the case $w(x)>w(y)$, $a=1$, $w(y)\mid w(x)$ above. An analoguos argument also
applies to the case $b\geq2$ and $a=1$.

\item[(1.f)] In the case $a=b=1$, a Gr\"{o}bner basis calculation shows
directly that any monomial corresponding to a lattice point of $\Delta$ is
either a monomimal of $f_{0}$ or an element of $B$ (note that smooth faces do
not contain any interior lattice points): The germ $x^{p}y+xy^{p}$, $p\geq2$,
is right-equivalent to a germ $h$ (applying for instance the right-equivalance
$x\mapsto x+y$, $y\mapsto y+2x$) with vertex monomials $x^{p+1}$ and $y^{p+1}%
$, hence, by Theorem has Milnor number $\mu=(p+1)^{2}-2(p+1)+1=p^{2}$. The
germ $g=x^{p}y+xy^{p}+x^{p^{2}+2}+y^{p^{2}+2}$ is right-equivalent to $h$,
hence also has Milnor number $\mu=p^{2}$. This implies that $\left\langle
x,y\right\rangle ^{p^{2}+2}\subset\operatorname{Jac}(g)$. We determine a
standard basis of
\[
\operatorname{Jac}(g)=\left\langle g_{x},g_{y}\right\rangle +\left\langle
x,y\right\rangle ^{p^{2}+2}%
\]
where%
\[
g_{x}=p\cdot x^{p-1}y+y^{p}+(p^{2}+2)\cdot x^{p^{2}+1}\text{, }g_{y}%
=x^{p}+p\cdot xy^{p-1}+(p^{2}+2)\cdot y^{p^{2}+1}%
\]
with regard the local degree reverse lexicographic ordering. The S-polynomial
of $g_{x}$ and $g_{y}$ leads to the standard basis element $xy^{p}$ after
reducing the tail by $\left\langle x,y\right\rangle ^{p^{2}+2}$. The
S-polynomial of $xy^{p}$ with $g_{x}$ leads to the standard basis element
$y^{2p-1}$ after reducing the tail by $\left\langle x,y\right\rangle
^{p^{2}+2}$. The S-polynomial of $xy^{p}$ with $g_{y}$ reduces to zero, while
all remaining ones vanish. Hence, the classes of the monomials%
\[
y^{p+1},x^{2}y^{p-1},\ldots,x^{p-2}y^{3}\text{,}%
\]
form a basis of the Milnor algebra in degree $p+1$. Using the relation $g_{x}$
these monomials are equivalent to%
\[
x^{2}y^{p-1},\ldots,x^{p-2}y^{3},x^{p-1}y^{2}%
\]
which form a regular basis in degree $p+1$. From the standard basis of the
Jacobian it is clear that this is the only option for a regular basis, which
proves the claim.
\end{itemize}\medskip

We now prove the second part of the lemma. Let $m_{1},\ldots,m_{n}$ be the
monomials of $f_{0}$ and let $g=\sum_{i}k_{i}m_{i}$. For a face $\Delta$ of
$\Gamma$, we denote again its weight by $w$. In the first part of the proof we
have seen that for all faces $\Delta$, except those where $a=1$, $w(x)\geq
w(y)$, $w(y)|w(x)$, and $b\geq2$ if $w(x)=w(y)$, or those where $b=1$,
$w(x)\leq w(y)$, $w(x)|w(y)$, and $b\geq2$ if $w(x)=w(y)$, there does not
exist a $w$-weighted homogeneous right-equivalence except rescaling of
variables.\footnote{Note that these cases correspond to the second part of
(1.c) and to (1.e), which are the only settings where there may exist (and, in
fact, exist) transformations on the jet of the respective face, which are not
just rescalings of variables.} In order to describe in these two cases the
$w$-homogeneous transformations on $\jet(g,\Delta)$ keeping the face, by
symmetry, it is sufficient to consider the first of the two. Here, writing
$\tau=w(x)/w(y)$, the jet is of the from $\jet(g,\Delta)=y^{q}\cdot\left(
k_{i}x^{p}+k_{j}xy^{(p-1)\cdot\tau}\right)  $, $k_{i},k_{j}\neq0$, $p\geq2$.
All homogeneous transformations are of the form $x\mapsto c_{1}x+c_{2}y^{\tau
},$ $y\mapsto c_{3}y$, and keeping $\Delta$ amounts to the condition%
\[
k_{i}c_{2}^{p}+k_{j}c_{2}c_{3}^{(p-1)\cdot\tau}=0\text{,}%
\]
which implies that either $c_{2}=0$ (which corresponds to a rescaling of
variables), or that between $c_{2}$ and $c_{3}$ there is an algebraic relation
$c_{2}=k\cdot c_{3}^{\tau}$ (with $k^{p-1}=-k_{j}/k_{i}$).

We now show that there is a germ $g$ with the same monomials as $f_{0}$ and
non-degenerate Newton boundary, such that only two terms of $g$ can be
normalized to coefficient one by a right-equivalence keeping the Newton
polygon. Since this $g$ is right-equivalent to a germ in the normal form
described in Theorem \ref{theorem:normalform}, this then implies that at most
two monomials of $f_{0}$ are not in $B$.

In case there are only two monomials of $f_{0}$ of degree $\leq dt$, the claim
is trivial choosing $g=f_{0}$.\footnote{Note that this includes the case where
there is a face with $w(x)=w(y)$ and $a=b=1$.} Otherwise, let $m_{s},m_{t}$
and $m_{l}$ be three distinct monomials of $f_{0}$ of degree $\leq dt$. We
prove that there is a Zariski open set of germs $g$ such that not all three
monomials $m_{s},m_{t},m_{l}$ can be normalized to coefficient one. To see
this, it is enough to prove that, for any $g$ with the same monomials as
$f_{0}$, after choosing two monomials out of $m_{s},m_{t},m_{l}$ and
normalizing their coefficients to one, when restricting the action of the
right-equivalence group to the Newton boundary and stabilizing the two
normalized coefficients, the stabilizer acts as a finite group.

\begin{itemize}[leftmargin=9mm]
\item[(2.a)] If all three monomials lie on faces of $\Gamma$ which permit only
rescaling of variables, then the claim is obvious.

\item[(2.b)] If exactly two of the three monomials, say $m_{s},m_{t}$, lie on
faces of $\Gamma$ which permit only rescaling of variables, then, without loss
of generality, we may assume that $m_{l}$ lies on a face with $a=1$, $w(x)\geq
w(y)$, $w(y)|w(x)$. If $w(x)=w(y)$, then it is sufficient to consider the case
$b\geq2$. After normalization of the coefficients of $m_{s}$ and $m_{t}$ to
value one via rescaling of variables, stabilizing the two normalized
coefficients together with the above condition $c_{2}=0$ or $c_{2}=k\cdot
c_{3}^{\tau}$ admits only finitely many solutions.

\item[(2.c)] Suppose that exactly one of the three monomials, say $m_{s}$,
lies on a face of $\Gamma$ which permits only rescaling of variables.

%\begin{itemize}
%\item 
If $m_{t}$ and $m_{l}$ lie on the same face $\Delta$, we may assume that
the jet of $\Delta$ is of the form $\jet(g,\Delta)=y^{q}\cdot\left(
k_{l}x^{p}+k_{t}xy^{(p-1)\cdot\tau}\right)  $. Then coefficients of the
monomials $m_{s}$ and $m_{l}$ can only be changed via rescaling of variables.
Then, similar to the previous case, normalization of the coefficients of
$m_{s}$ and $m_{l}$ to value one together with the condition $c_{2}=0$ or
$c_{2}=k\cdot c_{3}^{\tau}$ admits only finitely many solutions.

%\item 
If $m_{t}$ and $m_{l}$ lie on different faces, we may assume that
$m_{t}$ lies on a face $\Delta$ with weight $w$ and $a=1$, $w(x)\geq w(y)$,
$w(y)|w(x)$. Moreover, in case $w(x)=w(y)$, we can assume that $b\geq2$, since
the case $b=1$ has already been discussed. Then $\jet(g,\Delta)$ is of the
form
\[
\jet(g,\Delta)=y^{q}\cdot\left(  k_{i}x^{p}+k_{j}xy^{(p-1)\cdot\tau}\right)
\text{,}%
\]
with $\tau=w(x)/w(y)$ and $t\in\{i,j\}$, and right-equivalences keeping the
Newton polygon act on the jet as
\[
x\mapsto c_{1}x+c_{2}y^{\tau}\text{, }y\mapsto c_{3}y
\]
satisfy the condition $c_{2}=0$ or $c_{2}=k\cdot c_{3}^{\tau}$ with
$k^{p-1}=-k_{j}/k_{i}$. Similarly, we may assume that $m_{l}$ lies on a face
$\Pi$ with weight $v$ with $b=1$, $v(x)\leq v(y)$, $v(x)|v(y)$. Again, in case
$v(x)=v(y)$, we may assume that $a\geq2$. Then $\jet(g,\Pi)$ is of the form
\[
\jet(g,\Pi)=x^{q^{\prime}}\cdot\left(  k_{r}y^{p^{\prime}}+k_{s}%
yx^{(p^{\prime}-1)\cdot\tau^{\prime}}\right)  \text{,}%
\]
with $\tau^{\prime}=v(y)/v(x)$ and $l\in\{r,s\}$, and right-equivalences
keeping the Newton polygon act on the jet as
\[
x\mapsto c_{1}x\text{, }y\mapsto c_{3}y+c_{2}^{\prime}x^{\tau^{\prime}}%
\]
satisfying the condition $c_{2}^{\prime}=0$ or $c_{2}^{\prime}=k^{\prime}\cdot
c_{1}^{\tau^{\prime}}$ with $(k^{\prime})^{p-1}=-k_{s}/k_{r}$. If $m_{t}$ is
the monomial of $\jet(g,\Delta)$ of larger $x$-degree (that is, $m_{t}%
=y^{q}x^{p}$ and $t=i$), or if $m_{l}$ is the monomial of $\jet(g,\Pi)$ of
larger $y$-degree (that is, $m_{l}=y^{q^{\prime}}x^{p^{\prime}}$ and $l=r$),
then the coefficient of the respective monomial can only be changed via
rescaling of variables, and we can argue as in previous case. Suppose now that
$m_{t}$ and $m_{l}$ are the $x$-linear monomials of the respective jets.
Normalizing the coefficients of $m_{t}$ and $m_{l}$ then amounts to the
relations
\begin{align*}
c_{1}c_{3}^{q}\cdot(k_{t}c_{3}^{(p-1)\cdot\tau}+k_{i}c_{2}^{p-1})  &  =1\\
c_{3}c_{1}^{q^{\prime}}\cdot(k_{l}c_{1}^{(p^{\prime}-1)\cdot\tau^{\prime}%
}+k_{r}(c_{2}^{\prime})^{p^{\prime}-1})  &  =1\text{,}%
\end{align*}
After inserting $c_{2}=k\cdot c_{3}^{\tau}$ with $k^{p-1}=-k_{t}/k_{i}$ the
first equation implies that $c_{2}=0$, a contradiction. Similarly, inserting
$c_{2}^{\prime}=k^{\prime}\cdot c_{1}^{\tau^{\prime}}$ with $(k^{\prime
})^{p-1}=-k_{l}/k_{r}$ into the second equation, yields a contradiction.
Hence, $c_{2}=c_{2}^{\prime}=0$, that is, right-equivalences keeping the
Newton polygon act on the jets as rescaling of variables. So after normalizing
the coeffcients of two monomials, the coefficient of the third one can take
only finitely many values.

\item[(2.d)] If none of the three monomials $m_{s},m_{t}$ and $m_{l}$ lies on
a face of $\Gamma$ which permits only rescaling of variables, then we may
assume that $m_{s},m_{t}$ lie on a face with $a=1$, $w(x)\geq w(y)$,
$w(y)|w(x)$, and $b\geq2$ if $w(x)=w(y)$, and $m_{l}$ lies on a face with
$b=1$, $w(x)\leq w(y)$, $w(x)|w(y)$, and $a\geq2$ if $w(x)=w(y)$. We can thus
argue as previous cases to obtain only finitely many solutions for the action
of the right-equivalence group on the Newton boundary.
\end{itemize}
\end{proof}

\section{Next steps}
We have also developed an algorithm to enumerate all normal form families up to a specified Milnor number or modality. These algorithms will be presented in an up-coming paper.

\end{document}